\documentclass[12pt,british,a4paper,reqno]{amsart}
\usepackage[T1]{fontenc}
\usepackage[latin9]{inputenc}
\usepackage{geometry}
\usepackage{textcomp}
\usepackage{mathrsfs}
\usepackage{mathtools}
\usepackage{amsmath}
\usepackage{amsthm}
\usepackage{amssymb}
\usepackage{stmaryrd}
\usepackage{setspace}
\makeatletter
\numberwithin{figure}{section}
\numberwithin{equation}{section}
\theoremstyle{plain}
\newtheorem{thm}{\protect\theoremname}[section]
  \theoremstyle{definition}
  \newtheorem{defn}[thm]{\protect\definitionname}
  \theoremstyle{plain}
  \newtheorem{prop}[thm]{\protect\propositionname}
  \theoremstyle{remark}
  \newtheorem{rem}[thm]{\protect\remarkname}
  \theoremstyle{plain}
  \newtheorem{lem}[thm]{\protect\lemmaname}
  \theoremstyle{plain}
  \newtheorem{cor}[thm]{\protect\corollaryname}
  \theoremstyle{definition}
  \newtheorem{example}[thm]{\protect\examplename}

\usepackage{amsmath}
\usepackage{amssymb}
\usepackage{amsfonts}
\usepackage{braket}
\usepackage{fontenc}
\usepackage{ifsym}
\usepackage{proof}
\usepackage{enumerate}
\usepackage{graphicx}
\usepackage{qtree}
\usepackage{tikz}
\usepackage{tree-dvips}
\usepackage{tikz-cd}
\usepackage{hyperref}
\usepackage{bookmark}
\usepackage{fancyhdr}
\usepackage{mathrsfs}
\usepackage{color}
\usepackage{indentfirst}
\usepackage{mathtools}
\usepackage{tensor}

\makeatletter
\usetikzlibrary{matrix,arrows,decorations.pathmorphing,positioning,decorations.pathreplacing}
\tikzset{commutative diagrams/.cd, mysymbol/.style={start anchor=center,end anchor=center,draw=none}}
\makeatother

\newcommand{\commutes}[2][\circ]{\arrow[mysymbol]{#2}[description]{#1}}
\newcommand{\exact}[2][\square]{\arrow[mysymbol]{#2}[description]{#1}}
\newcommand{\PB}[2][\square]{\arrow[mysymbol,pos=0.27]{#2}[description]{#1}}
\newcommand{\PO}[2][\square]{\arrow[mysymbol,pos=0.27]{#2}[description]{#1}}

\makeatletter
\newcommand{\dotminus}{\mathbin{\text{\@dotminus}}}
\newcommand{\@dotminus}{%
  \ooalign{\hidewidth\raise1ex\hbox{.}\hidewidth\cr$\m@th-$\cr}}
\makeatother

\makeatletter
\newcommand*{\rom}[1]{\expandafter\@slowromancap\romannumeral #1@}
\makeatother

\newcommand{\CA}{\mathcal{A}}
\newcommand{\CB}{\mathcal{B}}
\newcommand{\CC}{\mathcal{C}}

\newcommand{\CH}{\mathcal{H}}
\newcommand{\CI}{\mathcal{I}}

\newcommand{\CP}{\mathcal{P}}

\newcommand{\CS}{\mathcal{S}}
\newcommand{\CT}{\mathcal{T}}
\newcommand{\CU}{\mathcal{U}}
\newcommand{\CV}{\mathcal{V}}

\newcommand{\CX}{\mathcal{X}}

\newcommand{\SF}{\mathscr{F}}
\newcommand{\SG}{\mathscr{G}}

\newcommand{\BA}{\mathbb{A}}
\newcommand{\BC}{\mathbb{C}}

\newcommand{\BR}{\mathbb{R}}
\newcommand{\BZ}{\mathbb{Z}}

\newcommand{\deff}{\coloneqq}

\newcommand{\Hom}{\operatorname{Hom}\nolimits}

\newcommand{\im}{\operatorname{im}\nolimits}

\newcommand{\der}{\mathsf{D}}
\newcommand{\kom}{\mathsf{K}}

\renewcommand{\leq}{\leqslant}

\renewcommand{\phi}{\varphi}

\newcommand{\iso}{\cong}

\newcommand{\into}{\hookrightarrow}
\newcommand{\onto}{\rightarrow\mathrel{\mkern-14mu}\rightarrow}

\bookmarksetup{numbered,open}

\newsavebox{\wideeqbox}

\let\amph=&

\newcommand{\binprod}{\mathbin\Pi}
\newcommand{\bincoprod}{\mathbin\amalg}
\newcommand{\ring}{k}
\newcommand{\sus}{\Sigma}
\newcommand{\add}{\operatorname{\mathsf{add}}\nolimits}
\newcommand{\End}{\operatorname{End}\nolimits}
\newcommand{\Aut}{\operatorname{Aut}\nolimits}
\newcommand{\Ext}{\operatorname{Ext}\nolimits}

\newcommand{\Ab}{\operatorname{\mathsf{Ab}}\nolimits}
\newcommand{\XR}{\operatorname{\mathcal{X}_\emph{R}}\nolimits}

\newcommand{\per}{\operatorname{\mathsf{per}}\nolimits}

\newcommand{\rad}{\operatorname{rad}\nolimits}

\newcommand{\op}{\mathrm{op}}

\newcommand{\Ker}{\operatorname{Ker}\nolimits}
\newcommand{\Cok}{\operatorname{Coker}\nolimits}
\newcommand{\cok}{\operatorname{coker}\nolimits}
\newcommand{\Image}{\operatorname{Im}\nolimits}
\newcommand{\Coim}{\operatorname{Coim}\nolimits}

\newcommand{\image}{\operatorname{im}\nolimits}
\newcommand{\coim}{\operatorname{coim}\nolimits}

\newcommand{\wt}[1]{\widetilde{#1}}
\newcommand{\wh}[1]{\widehat{#1}}
\newcommand{\rmod}[1]{\operatorname{\mathsf{mod}\,--}\nolimits{#1}}
\newcommand{\lmod}[1]{{#1}\operatorname{--\,\mathsf{mod}}\nolimits}

\newenvironment{acknowledgements}{\begin{abstract}} {\end{abstract}}

\makeatother

\usepackage{babel}
  \providecommand{\corollaryname}{Corollary}
  \providecommand{\definitionname}{Definition}
  \providecommand{\examplename}{Example}
  \providecommand{\lemmaname}{Lemma}
  \providecommand{\propositionname}{Proposition}
  \providecommand{\remarkname}{Remark}
\providecommand{\theoremname}{Theorem}

\begin{document}

\title{Auslander-Reiten theory in quasi-abelian and Krull-Schmidt categories\footnote{\url{https://doi.org/10.1016/j.jpaa.2019.04.017}}}

\author{Amit Shah}
\address{School of Mathematics \\ 
University of Leeds \\ 
Leeds, LS2 9JT \\ 
United Kingdom
}
\email{mmas@leeds.ac.uk}
\date{\today}

\begin{abstract}
We generalise some of the theory developed for abelian categories
in papers of Auslander and Reiten to semi-abelian and quasi-abelian
categories. In addition, we generalise some Auslander-Reiten
theory results of S. Liu for Krull-Schmidt categories by removing
the $\Hom$-finite and indecomposability restrictions. Finally, we
give equivalent characterisations of Auslander-Reiten sequences in
a quasi-abelian, Krull-Schmidt category.
\end{abstract}
\maketitle
%%%%%%%%%%%%%%%%%%%%%%%%%%%%%%%%%%%%%%%%%%%%
\section{\label{sec:Introduction}Introduction}

As is well-known, the work of Auslander and Reiten on \emph{almost split sequences}
(which later also became known as \emph{Auslander-Reiten
sequences}), introduced in \cite{AuslanderReiten-Rep-theory-of-Artin-algebras-III},
has played a large role in comprehending the representation theory
of artin algebras. In trying to understand these sequences, it became
apparent that two types of morphisms would also play a fundamental
role (see \cite{AuslanderReiten-Rep-theory-of-Artin-algebras-IV}).
Irreducible morphisms and minimal left/right almost split morphisms
(see Definitions \ref{def:irreducible morphism} and \ref{def:(min) left right almost split},
respectively) were defined in \cite{AuslanderReiten-Rep-theory-of-Artin-algebras-IV},
and the relationship between these morphisms and Auslander-Reiten
sequences was studied. In fact, many of the abstract results of Auslander
and Reiten were proven for an arbitrary abelian category, not just
a module category, and in this article we show that much of this Auslander-Reiten
theory also holds in a more general context---namely in
that of a quasi-abelian category.

A \emph{quasi-abelian} category is an additive category that has kernels
and cokernels, and in which kernels are stable under pushout and cokernels
are stable under pullback. Classical examples of such categories include:
any abelian category; the category of filtered modules over a filtered
ring; and the category of topological abelian groups. A modern example
has recently arisen from cluster theory as we recall now. Let $\CC$
denote the cluster category (see \cite{BMRRT-cluster-combinatorics},
\cite{CalderoChapotonSchiffler-quivers-arising-from-cluster-a_n-case}) associated to a finite-dimensional hereditary
$k$-algebra, where $k$ is a field. Fix a basic rigid object $R$
of $\CC$ and consider the \emph{partial cluster-tilted} algebra $\Lambda_{R}\deff(\End_{\CC}R)^{\op}$. In the study of the category
$\lmod{\Lambda_{R}}$ of finitely generated left $\Lambda_{R}$-modules,
the additive quotient $\CC/[\CX_{R}]$ has been useful, where $[\CX_{R}]$
is the ideal of morphisms factoring through objects of $\CX_{R}=\Ker(\Hom_{\CC}(R,-))$,
because a certain Gabriel-Zisman localisation (see \cite{GabrielZisman-calc-of-fractions})
of it is equivalent to $\lmod{\Lambda_{R}}$ (see \cite[Thm. 5.7]{BuanMarsh-BM2}).
We showed in \cite{Shah-quasi-abelian-hearts-of-twin-cotorsion-pairs-on-triangulated-cats}
that $\CC/[\CX_{R}]$ is a quasi-abelian category, and hence the generalisations
of the results of Auslander and Reiten that we prove in \S\ref{sec:Auslander-Reiten-theory-in-quasi-abelian-categories}
of this article can be used to more fully understand this category.

Furthermore, $\CC$, and hence $\CC/[\CX_{R}]$, is a $\Hom$-finite,
Krull-Schmidt category, and so one can utilise techniques from a different
perspective. Liu initiated the study of Auslander-Reiten theory in
$\Hom$-finite, Krull-Schmidt categories in \cite{LiuS-AR-theory-in-KS-cat}
and, in particular, introduced the notion of an admissible ideal (see
Definition \ref{def:admissible ideal of morphisms}). It was shown
in \cite[\S 1]{LiuS-AR-theory-in-KS-cat} that if $\CA$
is a $\Hom$-finite, Krull-Schmidt category and $\CI$ is an admissible
ideal of $\CA$, then, under suitable assumptions, irreducible morphisms
(between indecomposables) and minimal left/right almost split morphisms
behave well under the quotient functor $\CA\to\CA/\CI$. We extend
these results of Liu in the following ways: first, we are able to
remove the $\Hom$-finite assumption for all the results in \cite[\S 1]{LiuS-AR-theory-in-KS-cat};
and second, we are able to remove the indecomposability assumption
in \cite[Lem. 1.7 (1)]{LiuS-AR-theory-in-KS-cat}.
We also prove new observations in this setting, inspired by work of
Auslander and Reiten. Moreover, in case $\CA$ is quasi-abelian and
Krull-Schmidt we are able to provide equivalent criteria (see Theorem
\ref{thm: ASS IV.1.13 for KS quasi-abelian cat - xi X to Y to Z exact seq AR iff EndX local and g right almost split iff EndZ local and f right almost split iff f min left almost split iff g min right almost split iff EndX,EndZ local and f,g irreducible})
for when a short exact sequence in $\CA$ is an \emph{Auslander-Reiten
sequence} (in the sense of Definition \ref{def:AR-sequence in an additive category}).
In particular, this characterisation applies to the category $\CC/[\CX_{R}]$.

This article is structured as follows. In \S\ref{sec:prelims} we recall some background
material on preabelian categories, and in particular we remind the
reader on how one may define the first extension group for such categories.
In \S\ref{sec:Auslander-Reiten-theory-in-quasi-abelian-categories} we develop Auslander-Reiten theory for semi-abelian and quasi-abelian
categories. At the beginning of \S\ref{sec:Auslander-Reiten-Theory-in-Krull-Schmidt-Categories}, we switch focus to Auslander-Reiten
theory in Krull-Schmidt categories, and then present a characterisation
theorem for Auslander-Reiten sequences in a Krull-Schmidt, quasi-abelian
category. At the end of \S\ref{sec:Auslander-Reiten-Theory-in-Krull-Schmidt-Categories} we present an example of a $\Hom$-infinite, Krull-Schmidt category communicated to the author by P.-G. Plamondon. Lastly, in \S\ref{sec:example} we explore an example coming
from the cluster category as discussed above, which demonstrates the
theory we have developed in the earlier sections.

%%%%%%%%%%%%%%%%%%%%%%%%%%%%%%%%%%%%%%%%%%%%
\section{Preliminaries}\label{sec:prelims}

\subsection{Preabelian categories\label{subsec:Preabelian-categories}}

The categories we study in \S\ref{sec:Auslander-Reiten-theory-in-quasi-abelian-categories} are semi-abelian and quasi-abelian
categories (see Definitions \ref{def:semi-abelian} and \ref{def:quasi-abelian category},
respectively). A category of either kind is a preabelian category (see Definition
\ref{def:preabelian cat}) with some additional structure. In this
section, we recall the notion of a preabelian category, and provide
some basic results that will be helpful later. For more details, we
refer the reader to \cite{Rump-almost-abelian-cats}.
\begin{defn}
\label{def:preabelian cat}\cite[\S 5.4]{BucurDeleanu-intro-to-theory-of-cats-and-functors}
A \emph{preabelian} category is an additive category in which every
morphism has a kernel and a cokernel.
\end{defn}

\begin{rem}
\label{rem:preabelian cats have split idempotents}By \cite[Prop. 6.5.4]{Borceux-handbook-1},
any preabelian category $\CA$ \emph{has split idempotents}, or is
\emph{idempotent complete}, (see \cite[Def. 6.5.3]{Borceux-handbook-1})
because every morphism in $\CA$ admits a kernel, so in particular
every idempotent does. See also \cite[p. 188]{Auslander-Rep-theory-of-Artin-algebras-I-II}
and \cite[\S 6]{Buhler-exact-cats}.
\end{rem}

The following lemma is standard but often useful.
\begin{lem}
\label{lem:add cat with split idems, X indecomp and Y neq 0, then f X to Y retraction implies f isomorphism - and dual statement}Suppose
$\CA$ is an additive category with split idempotents. Let $f\colon X\to Y$
be a morphism in $\CA$. \emph{\begin{enumerate}[(i)]
\item \emph{Suppose $X$ is indecomposable and $Y\neq 0$. If $f$ is a retraction, then $f$ is an isomorphism.}
\item \emph{Suppose $Y$ is indecomposable and $X\neq 0$. If $f$ is a section, then $f$ is an isomorphism.}
\end{enumerate}}
\end{lem}

\begin{proof}
We only prove (i); the proof for (ii) is dual. Suppose $X$ is indecomposable,
$Y\neq0$ and that $f\colon X\to Y$ is a retraction. Then, by \cite[Rem. 7.4]{Buhler-exact-cats},
$X\iso Y\oplus Y'$ with $f$ corresponding to the canonical projection
$Y\oplus Y'\onto Y$. However, $X$ is indecomposable and $Y\neq0$
implies $Y'=0$, so $f$ is an isomorphism.
\end{proof}
The next lemma may be found as an exercise in \cite{Osborne-basic-homological-algebra},
and follows from \cite[Lem. IX.1.8]{Aluffi-Chapter0}: the proof in \cite{Aluffi-Chapter0}
is for the corresponding result in an abelian category, but is sufficient
for Lemma \ref{lem:ker is ker of its cok - cok is cok of its ker}
since only the existence of (co)kernels is needed.
\begin{lem}
\label{lem:ker is ker of its cok - cok is cok of its ker}\emph{\cite[Exer. 7.13]{Osborne-basic-homological-algebra}}
In a preabelian category: \emph{\begin{enumerate}[(i)]
\item \emph{every kernel is the kernel of its cokernel; and}
\item \emph{every cokernel is the cokernel of its kernel.}
\end{enumerate}}
\end{lem}

For a morphism $f\colon X\to Y$ in an additive category, we denote
the kernel (respectively, cokernel) of $f$, if it exists, by $\ker f\colon\Ker f\to X$
(respectively, $\cok f\colon Y\to\Cok f$).
\begin{lem}
\label{lem:epic kernel or monic cokernel is an isomorphism}Let $f\colon X\to Y$
be a morphism in a preabelian category. If $f$ is an epimorphism
and a kernel (respectively, a monomorphism and a cokernel), then $f$
is an isomorphism.
\end{lem}

\begin{proof}
We prove the statement for when $f$ is an epic kernel. The other
statement is dual.

Suppose $f\colon X\to Y$ is an epimorphism and a kernel. Note that
$f$ is monic by \cite[Lem. IX.1.4]{Aluffi-Chapter0}, and $\begin{tikzcd}[column sep=1.1cm]Y \arrow{r}{\cok f}& 0\end{tikzcd}$
is a cokernel of $f$ by \cite[Lem. IX.1.5]{Aluffi-Chapter0}. Now consider the identity morphism $1_{Y}\colon Y\to Y$, and notice that since $(\cok f)\circ1_{Y}=0\circ1_{Y}=0$
we have that $1_{Y}$ factors through $\ker(\cok f)=f$ (using Lemma
\ref{lem:ker is ker of its cok - cok is cok of its ker}). Thus, there
exists $g\colon Y\to X$ such that $fg=1_{Y}$. Thus, $f$ is a monic
retraction, and so an isomorphism by \cite[Thm. I.1.5]{McLarty-elementary-categories-elementary-toposes}.
\end{proof}
\begin{defn}
\label{def:image and coimage}\cite[p. 23]{Popescu-abelian-cats-with-apps-to-rings-and-modules}
Given a morphism $f\colon A\to B$ in an additive category, the \emph{coimage}
$\coim f\colon A\to\Coim f$, if it exists, is the cokernel $\cok(\ker f)$
of the kernel of $f$. Dually, the \emph{image} $\image f\colon\Image f\to B$
is the kernel $\ker(\cok f)$ of the cokernel of $f$.
\end{defn}

It is then immediate from Definition \ref{def:image and coimage}
that any image morphism is a monomorphism and any coimage is an epimorphism.
Furthermore, in a preabelian category $\CA$, any morphism $f$ admits
a factorisation as follows: $$\begin{tikzcd}[column sep=1.5cm] A \arrow[two heads]{d}[swap]{\coim f} \arrow{r}{f}\commutes{dr} & B \\ \Coim f  \arrow{r}[swap]{\widetilde{f}}& \Image f \arrow[hook]{u}[swap]{\image f}\end{tikzcd}$$where
$\widetilde{f}$ is known as the \emph{parallel of $f$}. This parallel
morphism is always an isomorphism in an abelian category, but this
may not be the case in general.

Recall that a sequence $X\overset{f}{\to}Y\overset{g}{\to}Z$ of morphisms
in an additive category is called \emph{short exact} if $f=\ker g$
and $g=\cok f$. We state a version of the well-known Splitting Lemma,
which is normally stated in the context of an abelian category (see,
for example, \cite[Prop. I.4.3]{MacLane-homology} or \cite[Prop. 1.8.7]{Borceux-handbook-2}),
for an additive category. Additionally, we do not assume initially
that the sequence of morphisms is short exact, since this is a consequence
of the equivalent conditions. We omit the proof because the one in
\cite{Borceux-handbook-2} works essentially unchanged.
\begin{prop}[Splitting Lemma]
\label{prop:Splitting-Lemma in additive category}Let $\CA$ be an
additive category with a sequence $X\overset{f}{\to}Y\overset{g}{\to}Z$
of morphisms. Then the following are equivalent.
\emph{\begin{enumerate}[(i)]
\item \emph{There is an isomorphism $Y\iso X\oplus Z$, where $f$ corresponds to the canonical inclusion $X\into X\oplus Z$ and $g$ to the canonical projection $X\oplus Z\onto Z$.}
\item \emph{The morphism $f$ is a section and $g=\cok f$.}
\item \emph{The morphism $g$ is a retraction and $f=\ker g$.}
\end{enumerate}}

In this case, $X\overset{f}{\to}Y\overset{g}{\to}Z$ is short exact.
\end{prop}

\begin{defn}
\label{def:non-split short exact seq}A short exact sequence $X\overset{f}{\to}Y\overset{g}{\to}Z$
in an additive category $\CA$ is called \emph{split} if it satisfies
any of the equivalent conditions of Proposition \ref{prop:Splitting-Lemma in additive category}.
Otherwise, the sequence is said to be \emph{non-split}.
\end{defn}

In a non-split short exact sequence $X\overset{f}{\to}Y\overset{g}{\to}Z,$
we have that $f$ is not a section and $g$ is not a retraction by
Proposition \ref{prop:Splitting-Lemma in additive category}. However,
more can be said as we see now.
\begin{lem}
\label{lem:non-split short exact sequence implies neither map is section or retraction}Let
$\CA$ be an additive category, and suppose $X\overset{f}{\to}Y\overset{g}{\to}Z$
is a non-split short exact sequence in $\CA$. Then both $f$ and
$g$ are neither sections nor retractions.
\end{lem}

\begin{proof}
As noted above, we need only show that $f$ is not a retraction and
that $g$ is not a section. Assume, for contradiction, that $f$ is
a retraction. Then $f$ is an epimorphism and so $Z\iso\Cok f\iso0$
by \cite[Lem. IX.1.5]{Aluffi-Chapter0}. However, this implies that $g$ is
a retraction which is a contradiction. Therefore, $f$ cannot be a
retraction. Showing $g$ is not a section is dual.
\end{proof}

%%%%%%%%%%%%%%%%%%%%%%%%%%%%%%%%%%%%%%%%%%%%
\subsection{\texorpdfstring{$\Ext$}{Ext} in a preabelian category\label{subsec:Ext-in-preabelian-cat}}

In order to avoid some $\Hom$-finiteness restrictions in later arguments,
we recall in this section how a first extension group (see Definition
\ref{def:Ext^1 in preabelian category}) may be defined in a preabelian
category in such a way that it is a bimodule (see Theorem \ref{thm:Ext^1 is a left-right EndX-EndZ bimodule in preabelian category}).
Although we follow the development in \cite{RichmanWalker-ext-in-preabelian-cats},
there is an error in their Theorem 4 (\cite[p. 523]{RichmanWalker-ext-in-preabelian-cats})
that is corrected in \cite{Cooper-stable-seqs-in-preabelian-cats}.
However, we also believe there should be more (set-theoretic) assumptions
in place to ensure that the first extension group is indeed a group
(see Remark \ref{rem:need Ext(Z X) to be a set to be a group}).

Throughout this section, $\CA$ denotes a preabelian category and
we suppose $X,Z$ are objects of $\CA$. We note for the next definition
that in a preabelian category pullbacks and pushouts always exist
as we have the existence of kernels and cokernels.
\begin{defn}
\label{def:cocomplex a xi and xi c for a cocomplex xi}\cite[p. 523]{RichmanWalker-ext-in-preabelian-cats}
Let $\xi\colon X\overset{f}{\to}Y\overset{g}{\to}Z$ be a \emph{complex}
(i.e. $g\circ f=0$) in $\CA$. Let $a\colon X\to X'$ be any morphism
in $\CA$. We define a new complex $a\xi$ as follows. First, we form
the pushout $X'\bincoprod_{X}Y$ of $a$ along $f$ with morphisms
$f'\colon X'\to X'\bincoprod_{X}Y$ and $a'\colon Y\to X'\bincoprod_{X}Y$.
Then we obtain a unique morphism $g'\colon X'\bincoprod_{X}Y\to Z$
using the universal property of the pushout with the morphisms $0\colon X'\to Z$
and $g\colon Y\to Z$ as in the following commutative diagram.$$\begin{tikzcd}X \arrow{r}{f}\arrow{d}{a}&Y\arrow{r}{g}\arrow{d}{a'} & Z\arrow[equals]{dd} \\
X'\arrow{r}{f'}\arrow[bend right]{drr}[swap]{0} &X'\bincoprod_{X}Y\arrow[dotted]{dr}{\exists !g'}\PO{ul}&\\
&&Z \end{tikzcd}$$The complex $a\xi$ is then defined to be $X'\overset{f'}{\to}X'\bincoprod_{X}Y\overset{g'}{\to}Z.$

Dually, we define $\xi b$ for a morphism $b\colon Z'\to Z$. The
commutative diagram$$\begin{tikzcd}X\arrow[dotted]{dr}{\exists ! f'}\arrow[equals]{dd}\arrow[bend left]{drr}{0} && \\ 
& Y \binprod_Z Z'\arrow{r}{g'}\arrow{d}{b'}\PB{dr}& Z'\arrow{d}{b}\\
X\arrow{r}{f} & Y\arrow{r}{g} & Z\end{tikzcd}$$summarises the construction and $\xi b$ is the complex $X\overset{f'}{\to}Y\binprod_{Z}Z'\overset{g'}{\to}Z'$.
\end{defn}

The next definition is standard terminology.
\begin{defn}
\label{def:stable under PB (PO)}Suppose $$\begin{tikzcd}A\arrow{r}{a}\arrow{d}[swap]{b} & B \arrow{d}{c}\\ C \arrow{r}[swap]{d} & D\end{tikzcd}$$is
a commutative diagram in a category $\CA$. Let $\CP$ be a class of morphisms in $\CA$ (e.g. the class of all kernels in $\CA$). We say that $\CP$ is
\emph{stable under pullback} (respectively, \emph{stable under pushout})
if, in any diagram above that is a pullback (respectively, pushout),
$d$ is in $\CP$ implies $a$ is in $\CP$ (respectively, $a$ is in $\CP$
implies $d$ is in $\CP$).
\end{defn}

In a preabelian category, kernels are stable under pullback (see \cite[Thm. 1]{RichmanWalker-ext-in-preabelian-cats}),
but they may not be stable under pushout. Dually for cokernels. Thus,
Richman and Walker make the following definition.
\begin{defn}
\label{def:stable exact sequence}\cite[p. 524]{RichmanWalker-ext-in-preabelian-cats}
Let $\xi\colon X\overset{f}{\to}Y\overset{g}{\to}Z$ be a short exact
sequence in $\CA$. We say that $\xi$ is \emph{stable}, if $a\xi$
and $\xi b$ are short exact for all $a\colon X\to X',b\colon Z'\to Z$.
In this case, we call $f=\ker g$ a \emph{stable kernel} and $g=\cok f$
a \emph{stable cokernel}. We will sometimes also call $\xi$ \emph{stable
exact} in this case to emphasise the exactness of $\xi$.
\end{defn}

Suppose $\nu\colon A\overset{a}{\to}B\overset{b}{\to}C$ and $\xi\colon X\overset{f}{\to}Y\overset{g}{\to}Z$
are two short exact sequences in $\CA$. Recall that a \emph{morphism 
$(u,v,w)\colon\nu\to\xi$ of short exact sequences} is a
commutative diagram $$\begin{tikzcd}
A\arrow{r}{a}\arrow{d}{u} & B\arrow{r}{b}\arrow{d}{v} & C\arrow{d}{w}\\
X\arrow{r}{f} & Y\arrow{r}{g} & Z
\end{tikzcd}$$in $\CA$. If $A=X$ and $C=Z$, then a morphism of the form $(1_{X},v,1_{Z})$
in which $v\colon B\to Y$ is an isomorphism is called an \emph{isomorphism
of short exact sequences with the same end-terms}, and we denote this
by $\nu\tensor[_{X}]{\iso}{_{Z}}\xi$. This is clearly an equivalence
relation on the class of short exact sequences of the form $X\to-\to Z$.
\begin{thm}
\label{thm:C80 Thm 2 stable exact implies associativity ie bimodule structure for Ext}\emph{\cite[Thm. 2]{Cooper-stable-seqs-in-preabelian-cats}}
Suppose $\xi\colon X\overset{f}{\to}Y\overset{g}{\to}Z$ is a stable
exact sequence in $\CA$. Then there is an isomorphism $a(\xi b)\tensor[_{X'}]{\iso}{_{Z'}}(a\xi)b$
for all $a\colon X\to X',b\colon Z'\to Z$ in $\CA$.
\end{thm}

\begin{rem}
It can readily be seen that a statement like Theorem \ref{thm:C80 Thm 2 stable exact implies associativity ie bimodule structure for Ext}
is needed to give a definition of an extension group $\Ext_{\CA}^{1}(Z,X)$
that is also an $(\End_{\CA}X,\End_{\CA}Z)$-bimodule. Theorem \ref{thm:C80 Thm 2 stable exact implies associativity ie bimodule structure for Ext}
was claimed to hold for all sequences in \cite{RichmanWalker-ext-in-preabelian-cats}
(see \cite[Thm. 4]{RichmanWalker-ext-in-preabelian-cats}),
but a counterexample was given in \cite{Cooper-stable-seqs-in-preabelian-cats}.
Cooper presents the corrected statement as \cite[Thm. 2]{Cooper-stable-seqs-in-preabelian-cats},
presenting one half of the argument and suggesting a diagram to use
for the dual argument. 
%However, the suggested diagram
%\[
%\alpha(\xi\beta)\to(\alpha\xi)\beta\to\alpha\xi\to f_{3}\alpha\xi
%\]
%in the proof of \cite[Thm. 2]{Cooper-stable-seqs-in-preabelian-cats}
%should be replaced by
%\[
%\alpha(\xi\beta)\to(\alpha\xi)\beta\to\alpha\xi\to f_{2}\alpha\xi.
%\]
%It easy to find a morphism $\alpha(\xi\beta)\to(\alpha\xi)\beta$
%of the form $(1_{A'},\phi,1_{C'})$, but to show that $\phi$ is an
%isomorphism requires the stability of $\xi$ (see \cite[Cor. 7]{RichmanWalker-ext-in-preabelian-cats}).
%In an abelian category, the fact that $\phi$ is an isomorphism would
%follow quickly, for example, from the Five Lemma.
However, the suggested dual diagram is not the right one to consider. 

First, let us recall the setup in \cite[Thm. 2]{Cooper-stable-seqs-in-preabelian-cats}. Let $E\colon A\overset{f}{\to} B \overset{g}{\to} C$ be a stable exact sequence, and suppose we have morphisms $\alpha\colon A\to A'$ and $\beta\colon C'\to C$. Then, as obtained in \cite[p. 266]{Cooper-stable-seqs-in-preabelian-cats}, there is a commutative diagram
$$\hspace{-3cm}\begin{tikzcd}
\alpha(E\beta)\colon \arrow{d}[swap]{(1_{A'},\phi_2,1_{C'})}& A' \arrow[equals]{d}\arrow{r}{f_2}& B_2 \arrow{r}{g_2}\arrow{d}{\phi_2}& C'\arrow[equals]{d} \\
(\alpha E)\beta\colon & A' \arrow{r}{f_3}& B_3\arrow{r}{g_3} & C'
\end{tikzcd}$$

It is shown in detail that $g_3=\cok f_3$ and $\phi_2$ is an epimorphism. It is then suggested that the diagram 
$$\alpha(E\beta)\to(\alpha E)\beta\to\alpha E\to f_{3}\alpha E$$ 
with a dual proof strategy will yield $f_2=\ker g_2$ and $\phi_2$ is a monomorphism. However, this diagram should be replaced by 
$$\alpha(E\beta)\to(\alpha E)\beta\to\alpha E\to f_{2}\alpha E.$$ 

Furthermore, we note that it is straightforward to find a morphism $\alpha(E\beta)\to(\alpha E)\beta$ of the form $(1_{A'},\phi_2,1_{C'})$, but to then show that $\phi_2$ is an isomorphism requires the stability of $E$ (see \cite[Cor. 7]{RichmanWalker-ext-in-preabelian-cats}). In an abelian category, the fact that $\phi_2$ is an isomorphism would follow quickly, for example, from the Five Lemma.
\end{rem}

Following \cite{MacLane-homology}, we introduce some notation
to help the reading of the sequel. Let $A,B,C,D$ be objects in $\CA$.
We denote by $\nabla_{A}$ the \emph{codiagonal} morphism $(\,1_{A}\,\;1_{A}\,)\colon A\oplus A\to A$, and denote by $\Delta_{A}$ the \emph{diagonal} morphism $\begin{psmallmatrix}1_A \\ 1_A\end{psmallmatrix}\colon A\to A\oplus A.$
For two morphisms $a\colon A\to C$ and $b\colon B\to D,$ we let
$a\oplus b$ denote the morphism $\begin{psmallmatrix}a & 0\\0&b\end{psmallmatrix}\colon A\oplus B\to C\oplus D.$
We are now in a position to define $\Ext_{\CA}^{1}(Z,X)$.
\begin{defn}
\label{def:Ext^1 in preabelian category}\cite[\S 4]{RichmanWalker-ext-in-preabelian-cats}
Let $\CA$ be a preabelian category. 
%The \emph{first extension group}
%is defined to be the class $\Ext_{\CA}^{1}(Z,X)$ of equivalence classes
%under $\tensor[_{X}]{\iso}{_{Z}}$ of stable short exact sequences
%of the form $X\to-\to Z$ in $\CA$, with the usual Baer sum 
%as the group operation that we recall here.
Define $\Ext_{\CA}^{1}(Z,X)$ to be the class of equivalence classes
under $\tensor[_{X}]{\iso}{_{Z}}$ of stable short exact sequences
of the form $X\to-\to Z$ in $\CA$.

By abuse of terminology/notation, by an element $\xi$ of $\Ext_{\CA}^{1}(Z,X)$
we will really mean the equivalence class $[\xi]_{\tensor[_{X}]{\iso}{_{Z}}}$
of $\xi$ in $\Ext_{\CA}^{1}(Z,X)$. If $\xi\colon X\overset{f}{\to}Y\overset{g}{\to}Z$,
$\xi'\colon X\overset{f'}{\to}Y'\overset{g'}{\to}Z$ are elements
of $\Ext_{\CA}^{1}(Z,X)$, then we define the \emph{Baer sum} of $\xi$
and $\xi'$ to be (the equivalence class of) $$\xi+\xi'\deff\nabla_{X}(\xi\oplus\xi')\Delta_{Z}.$$Note
that by \cite[Thm. 8]{RichmanWalker-ext-in-preabelian-cats}
and \cite[Thm. 2]{Cooper-stable-seqs-in-preabelian-cats},
$\xi+\xi'$ is stable exact and the Baer sum $+$ is a closed binary
operation on $\Ext_{\CA}^{1}(Z,X)$.
\end{defn}

\begin{rem}
\label{rem:need Ext(Z X) to be a set to be a group}It is observed
in \cite{RichmanWalker-ext-in-preabelian-cats} that one may
then follow \cite[pp. 70--71]{MacLane-homology} in order to show
that $\Ext_{\CA}^{1}(Z,X)$ is an abelian group. However, $\Ext_{\CA}^{1}(Z,X)$
may not form a (small) set and hence may not be a group. A similar issue arises in \cite{Cooper-stable-seqs-in-preabelian-cats}.

Note, however, that if $\CA$ is skeletally small, then $\Ext_{\CA}^{1}(Z,X)$
will be a set. Indeed, for objects $Y,Y'$ in $\CA$, the sets $\{\xi\colon X\to Y\to Z\mid\xi\text{ is short exact}\}$
and $\{\xi'\colon X\to Y'\to Z\mid\xi'\text{ is short exact}\}$ are
in bijection whenever $Y$ is isomorphic to $Y'$. So, up to equivalence
with respect to $\tensor[_{X}]{\iso}{_{Z}}$, the collection of all
short exact sequences of the form $X\overset{f}{\to}Y\overset{g}{\to}Z$
is determined only by the isomorphism class of $Y$ and the morphisms
$f,g$ since the end-terms $X$ and $Z$ are fixed. Therefore, the
collection of all $\tensor[_{X}]{\iso}{_{Z}}$-equivalence classes
will form a set, and hence restricting our attention to the classes
of stable exact sequences will also yield a set.
\end{rem}

These set-theoretic considerations lead us to the next theorem.
\begin{thm}
\emph{\label{thm:Ext^1 in preabelian cat is an abelian group}\cite[\S 4]{RichmanWalker-ext-in-preabelian-cats}}
Suppose $\CA$ is a preabelian category with objects $X,Z$, and suppose
$\Ext_{\CA}^{1}(Z,X)$ is a set. Then $\Ext_{\CA}^{1}(Z,X)$ is an
abelian group with the group operation given by the Baer sum defined
in \emph{Definition \ref{def:Ext^1 in preabelian category}}. The
class of the split extension $\xi_{0}\colon X\to X\oplus Z\to Z$
is the identity element, and the inverse of $\xi\in\Ext_{\CA}^{1}(Z,X)$
is $(-1_{X})\xi$.
\end{thm}
Therefore, if $\CA$ is preabelian, $\Ext_{\CA}^{1}(Z,X)$ is known as a \emph{first extension group}. 
We state without proof one corresponding result from \cite{MacLane-homology}
that is needed for the last theorem of this section.
\begin{lem}
\label{lem:Mac p71 0xi iso split and 1xi is xi and dual statements on other side}\emph{\cite[p. 71]{MacLane-homology}}
Let $\xi_{0}$ denote the split short exact sequence $X\to X\oplus Z\to Z$
and let $\xi\colon X\overset{f}{\to}Y\overset{g}{\to}Z$ be an arbitrary
stable exact sequence. Let $0_{X}$ (respectively, $0_{Z}$) denote
the zero morphism in $\End_{\CA}X$ (respectively, $\End_{\CA}Z$).
Then $0_{X}\xi\tensor[_{X}]{\iso}{_{Z}}\xi_{0}$ and $1_{X}\xi=\xi$,
and $\xi0_{Z}\tensor[_{X}]{\iso}{_{Z}}\xi_{0}$ and $\xi1_{Z}=\xi.$
\end{lem}

\begin{thm}
\label{thm:Ext^1 is a left-right EndX-EndZ bimodule in preabelian category}Let
$\CA$ be a preabelian category with objects $X,Z$, and suppose $\Ext_{\CA}^{1}(Z,X)$
is a set. Then $\Ext_{\CA}^{1}(Z,X)$ is an $(\End_{\CA}X,\End_{\CA}Z)$-bimodule.
\end{thm}

\begin{proof}
This follows from Theorem \ref{thm:Ext^1 in preabelian cat is an abelian group},
\cite[Thm. 4]{RichmanWalker-ext-in-preabelian-cats}, Lemma
\ref{lem:Mac p71 0xi iso split and 1xi is xi and dual statements on other side}
and Theorem \ref{thm:C80 Thm 2 stable exact implies associativity ie bimodule structure for Ext}.
\end{proof}

%%%%%%%%%%%%%%%%%%%%%%%%%%%%%%%%%%%%%%%%%%%%
\section{\label{sec:Auslander-Reiten-theory-in-quasi-abelian-categories}Auslander-Reiten
theory in quasi-abelian categories}

In this section, after recalling the definitions of a semi-abelian
and a quasi-abelian category, we will explore some Auslander-Reiten
theory type results in connection with these categories. We remark
here that quasi-abelian categories carry a canonical exact structure:
a quasi-abelian category $\CA$ endowed with the class of all short
exact sequences forms an exact category in the sense of Quillen \cite{Quillen-higher-algebraic-k-theory-I}
(see \cite[Rem. 1.1.11]{Schneiders-quasi-abelian-cats-and-sheaves}).
Some Auslander-Reiten theory for exact categories was developed in
\cite{GabrielRoiter-reps-of-finite-dimensional-algebras}, but
our results are different in nature: we explore properties of the
morphisms involved in Auslander-Reiten sequences, whereas \cite{GabrielRoiter-reps-of-finite-dimensional-algebras}
focuses more on the existence and construction of such sequences.
See \cite{IyamaNakaokaPalu-Auslander-Reiten-theory-in-extriangulated-categories}
also.
\begin{defn}
\label{def:semi-abelian}\cite[p. 167]{Rump-almost-abelian-cats}
Let $\CA$ be a preabelian category. We call $\CA$ \emph{left semi-abelian}
if each morphism $f\colon A\to B$ factorises as $f=ip$ for some
monomorphism $i$ and cokernel $p$. Similarly, $\CA$ is said to
be \emph{right semi-abelian} if instead each morphism $f$ decomposes
as $f=ip$ with $i$ a kernel and $p$ some epimorphism. If $\CA$
is both left and right semi-abelian, then it is called \emph{semi-abelian}.
\end{defn}

Note that a preabelian category is semi-abelian if and only if, for
every morphism $f$, the parallel $\widetilde{f}$ (see \S\ref{subsec:Preabelian-categories})
of $f$ is \emph{regular}, i.e. both monic and epic (see \cite[pp. 167--168]{Rump-almost-abelian-cats}).
\begin{defn}
\label{def:quasi-abelian category}\cite[p. 168]{Rump-almost-abelian-cats}
Let $\CA$ be a preabelian category. We call $\CA$ \emph{left quasi-abelian}
if cokernels are stable under pullback (see Definition \ref{def:stable under PB (PO)}) in $\CA$.
If kernels are stable under pushout in $\CA$, then we call $\CA$ \emph{right
quasi-abelian}. If $\CA$ is left and right quasi-abelian, then $\CA$
is simply called \emph{quasi-abelian}.
\end{defn}

\begin{example}
Any abelian category is quasi-abelian.
\end{example}

\begin{example}
The category of \emph{Banach spaces}, i.e. complete normed vector
spaces, over $\BR$ or $\BC$ is quasi-abelian, but not abelian; see \cite[p. 214]{Rump-almost-abelian-cats}.
\end{example}

Quasi-abelian categories, as we define them here, were called `almost
abelian' categories in \cite{Rump-almost-abelian-cats},
but the terminology we adopt is the more widely accepted one. See
the `Historical remark' in \cite{Rump-counterexample-to-Raikov}
for more details.
\begin{rem}
\label{rem:quasi-abelian-is-semi-abelian-and-decomposition-of-morphism}Rump
shows that every left (respectively, right) quasi-abelian category
is left (respectively, right) semi-abelian (see \cite[p.129, Cor. 1]{Rump-almost-abelian-cats}).
Furthermore, in a left (respectively, right) semi-abelian category,
if a morphism $f$ factorises as $f=ip$ with $i$ monic and $p$
a cokernel (respectively, $i$ a kernel and $p$ epic), then $p=\coim f$
(respectively, $i=\image f$) up to unique isomorphism.
\end{rem}

Auslander and Reiten showed that irreducible and (minimal) left/right
almost split morphisms (introduced in \cite{AuslanderReiten-Rep-theory-of-Artin-algebras-IV})
play a large role in the study of almost split sequences (defined
in \cite{AuslanderReiten-Rep-theory-of-Artin-algebras-III})
in abelian categories. The same is true in the generality we consider
in this article, and we begin by recalling the definition of an irreducible
morphism.
\begin{defn}
\label{def:irreducible morphism}\cite[\S 2]{AuslanderReiten-Rep-theory-of-Artin-algebras-IV}
A morphism $f\colon X\to Y$ of an arbitrary category is \emph{irreducible}
if the following conditions are satisfied:\begin{enumerate}[(i)]
\item $f$ is not a section;
\item $f$ is not a retraction; and
\item if $f=hg$, for some $g\colon X\to Z$ and $h\colon Z\to Y$, then either $h$ is a retraction or $g$ is a section.
\end{enumerate}
\end{defn}

For the results presented here that are analogues of those in known
work, we omit the proofs that carry over or that are easy generalisations.
Instead, we focus on those arguments that need significant modification
or that have been omitted in previous work. Furthermore, many of the
results in the remainder of the article have duals, which we state
but do not prove.

The next proposition is a version of \cite[Prop. 2.6 (a)]{AuslanderReiten-Rep-theory-of-Artin-algebras-IV}
for the semi-abelian setting. Recall that a monomorphism (respectively,
epimorphism) that is not an isomorphism is called a \emph{proper monomorphism} (respectively, \emph{proper epimorphism}).
\begin{prop}
\label{prop:Left or right semi-abelian implies irred map is either mono or epi}Suppose
a category $\CA$ is left or right semi-abelian. If $f\colon X\to Y$
is irreducible in $\CA$, then it is a proper monomorphism or a proper
epimorphism.
\end{prop}

\begin{proof}
Suppose $f\colon X\to Y$ is an irreducible morphism with coimage
$\coim f\colon X\onto M.$ Note that $f$ cannot be an isomorphism
since it is not, for example, a section.

Suppose now that $\CA$ is left semi-abelian. Then we have a factorisation
$f=i\circ\coim f$ where $i$ is monic (see Remark \ref{rem:quasi-abelian-is-semi-abelian-and-decomposition-of-morphism}).
If $f$ is a proper monomorphism then we are done, so suppose not.
Then $f=i\circ\coim f$ is irreducible implies $i$ is a retraction
or $\coim f$ is a section. The latter implies $\coim f$ is monic
and this in turn yields that $f$ is monic, which is contrary to our
assumption that $f$ is not a proper monomorphism. Thus, $i$ must
be a retraction and hence an epimorphism. Then $f$ is the composition
of two epimorphisms and is thus epic itself, i.e. $f$ is a proper
epimorphism.

The case when $\CA$ is right semi-abelian is proved similarly.
\end{proof}
In an abelian category, we have that a morphism $f$ is an isomorphism
if and only if it is regular. Therefore, in an abelian setting an
irreducible morphism is either a proper monomorphism or a proper epimorphism,
but never both simultaneously. However, this need not hold in an arbitrary
category. In particular, we will see in Example \ref{exa:Q is A_3 R is P_1 oplus P_2}
irreducible morphisms (so non-isomorphisms) that are regular.

The following two results give a version of \cite[Prop. 2.7]{AuslanderReiten-Rep-theory-of-Artin-algebras-IV}
in a more general setting. For the first proposition, we only really
need that the category $\CA$ is right semi-abelian and left quasi-abelian,
but by \cite[Prop. 3]{Rump-almost-abelian-cats} this
is equivalent to $\CA$ being (left and right) quasi-abelian because
left quasi-abelian implies left semi-abelian. Dually, for the second
result we only require that $\CA$ is left semi-abelian and right
quasi-abelian. The proof we give is inspired by that of Auslander
and Reiten; however, since regular morphisms may not be isomorphisms
or, for example, monomorphisms may not be kernels in the categories
we are dealing with, we must consider some different short exact sequences
in the proof.
\begin{prop}
\label{prop:ASS Lem IV.1.7(a) or AR IV Prop 2.7 (a) for quasi-abelian category - criterion for irreducible if f is a kernel}Suppose
$\CA$ is a quasi-abelian category and that $f\colon X\to Y$ is a
morphism in $\CA$ with cokernel $c\colon Y\to C$. If $f$ is irreducible,
then for all $v\colon V\to C$ either there exists $v_{1}\colon V\to Y$
such that $cv_{1}=v$ or there exists $v_{2}\colon Y\to V$ such that
$c=vv_{2}$. Furthermore, if $X\overset{f}{\to}Y\overset{c}{\to}C$
is a non-split short exact sequence, then the converse also holds.
\end{prop}

\begin{proof}
First, suppose that $f\colon X\to Y$ is irreducible and that $v\colon V\to C$
is arbitrary. Since $\CA$ is quasi-abelian we may consider the following
commutative diagram$$\begin{tikzcd} L \arrow[equal]{d} \arrow{r}{\ker a} & E \arrow{r}{a} \arrow{d}{u} \PB{dr}& V \arrow{d}{v} \\
L \arrow{r}{\im f} & Y \arrow{r}{c} & C \\
X \arrow[two heads]{u}{p} \arrow{ur}{f}& & \end{tikzcd}$$where $E$ is the pullback of $c$ along $v$, $f=(\im f)\circ p$ with
$p$ epic, and $a$ is a cokernel since cokernels are stable under
pullback in (left) quasi-abelian categories. Then $f=(\im f)\circ p=u\circ((\ker a)p)$,
so either $(\ker a)p$ is a section or $u$ is a retraction as $f$
is irreducible.

If $(\ker a)p$ is a section then there exists $v_{1}'\colon E\to X$
such that $v_{1}'(\ker a)p=1_{X}$. Then $pv_{1}'(\ker a)p=p=1_{L}p$,
so $pv_{1}'\ker a=1_{L}$ as $p$ is epic. Since $a$ is a cokernel,
we have $a=\cok(\ker a)$ by Lemma \ref{lem:ker is ker of its cok - cok is cok of its ker}.
Therefore, $\ker a$ is a section implies $a$ is a retraction by
the Splitting Lemma (Proposition \ref{prop:Splitting-Lemma in additive category}).
That is, there exists $v_{1}''\colon V\to E$ such that $av_{1}''=1_{V}$.
Now define $v_{1}\deff uv_{1}''$ and note that $cv_{1}=cuv_{1}''=vav_{1}''=v$.
Otherwise, $u$ is a retraction and so there exists $v_{2}'\colon Y\to E$
with $uv_{2}'=1_{Y}$. Setting $v_{2}\deff av_{2}'$ we see that $vv_{2}=vav_{2}'=cuv_{2}'=c$.
This concludes the proof of the first statement.

For the converse, we assume that $X\overset{f}{\to}Y\overset{c}{\to}C$
is a non-split short exact sequence and, further, that for all $v\colon V\to C$
either there exists $v_{1}\colon V\to Y$ such that $cv_{1}=v$ or
there exists $v_{2}\colon Y\to V$ such that $c=vv_{2}$. Then $f$
is not a section or a retraction, by Lemma \ref{lem:non-split short exact sequence implies neither map is section or retraction},
as $X\overset{f}{\to}Y\overset{c}{\to}C$ is non-split. It remains
to show part (iii) of Definition \ref{def:irreducible morphism}.
To this end, suppose $f=hg$ for some $g\colon X\to U$ and $h\colon U\to Y$.
Since $hg=f=\ker c$ is a kernel, $g$ is also a kernel by \cite[Prop. 2]{Rump-almost-abelian-cats}
as $\CA$ is quasi-abelian and so, in particular, right semi-abelian.
Thus, $g=\ker(\cok g)$ (by Lemma \ref{lem:ker is ker of its cok - cok is cok of its ker})
and $\begin{tikzcd}[column sep=1.1cm]X\arrow{r}{g}& U\arrow{r}{\cok g}& V\end{tikzcd}$
is short exact. Consider the commutative diagram$$\begin{tikzcd}[column sep=1.3cm] X \arrow[equal]{d} \arrow{r}{g} & U \arrow{r}{\cok g} \arrow{d}{h} & V \arrow{d}{v} \\
X \arrow{r}{f} & Y \arrow{r}{c} & C \end{tikzcd}$$where $v$ exists since $(ch)g=cf=0$. Since $1_{X}$ is an isomorphism,
$\cok g$ and $c$ are cokernels, and $\CA$ is quasi-abelian, by
\cite[Prop. 5]{Rump-almost-abelian-cats} we have that
the right square is \emph{exact}, i.e. simultaneously a pullback and
a pushout. Then, by assumption, either there exists $v_{1}\colon V\to Y$
such that $cv_{1}=v$ or there exists $v_{2}\colon Y\to V$ such that
$c=vv_{2}$. In the first case, we have the following situation$$\begin{tikzcd}[column sep=1.3cm] V\arrow[bend right]{ddr}[swap]{v_1} \arrow[bend left]{drr}{1_V}\arrow[dotted]{dr}{\exists !a}&&\\
&U \arrow{r}{\cok g} \arrow{d}[swap]{h} \exact{dr}& V \arrow{d}{v}\\
&Y \arrow{r}[swap]{c} & C \end{tikzcd}$$since $cv_{1}=v=v1_{V}$, and hence there exists (a unique) $a\colon V\to U$
such that $(\cok g)\circ a=1_{V}$ (and $v_{1}=ha$). Thus, $\cok g$
is a retraction and by the Splitting Lemma (Proposition \ref{prop:Splitting-Lemma in additive category})
we have that $g$ is a section. Otherwise, in the case where $v_{2}$
exists, we have that there is (a unique) $b\colon Y\to U$ such that
$hb=1_{Y}$ (and $(\cok g)\circ b=v_{2}$), in which case $h$ is
seen to be a retraction. The following diagram summarises this case.$$\begin{tikzcd}[column sep=1.3cm] Y\arrow[bend right]{ddr}[swap]{1_Y} \arrow[bend left]{drr}{v_2}\arrow[dotted]{dr}{\exists !b}&&\\
&U \arrow{r}{\cok g} \arrow{d}[swap]{h} \exact{dr}& V \arrow{d}{v}\\
&Y \arrow{r}[swap]{c} & C \end{tikzcd}$$Therefore, $f$ is irreducible and the proof is complete.
\end{proof}
The dual statement is as follows.
\begin{prop}
\label{prop:ASS Lem IV.1.7(b) or AR IV Prop 2.7 (b) for quasi-abelian category - criterion for irreducible if f is a cokernel}Suppose
$\CA$ is a quasi-abelian category. Suppose $f\colon X\to Y$ is a
morphism in $\CA$ with kernel $\ker f\colon \Ker f\to X$. If $f$ is irreducible,
then for all $u\colon \Ker f\to U$ either there exists $u_{1}\colon X\to U$
such that $u_{1}\ker f=u$ or there exists $u_{2}\colon U\to X$ such that
$\ker f=u_{2}u$. Furthermore, if $\begin{tikzcd}\Ker f \arrow{r}{\ker f}& X \arrow{r}{f}& Y\end{tikzcd}$ is a non-split short exact sequence, then the converse also holds.
\end{prop}

Let $\ring$ be a commutative (unital) ring and suppose $\CA$ is a \emph{$\ring$-category}, i.e. an additive category in which the set of morphisms
between any two objects is a $\ring$-module and composition of morphisms
is $\ring$-bilinear. The \emph{radical} $\rad_{\CA}(-,-)$ of a $\ring$-category
is the (two-sided) ideal of $\CA$ defined by 
\[
\rad_{\CA}(X,Y)\deff\{f\in\Hom_{\CA}(X,Y)\mid1_{X}-gf\text{ is invertible for all }g\colon Y\to X\}
\]
for any two objects $X,Y\in\CA$. By a \emph{radical} morphism $f\colon X\to Y$,
we mean an element of $\rad_{\CA}(X,Y)$. Furthermore, $\rad_{\CA}(X,X)\subseteq\End_{\CA}X$
coincides with the Jacobson radical $J(\End_{\CA}X)$ of the ring $\End_{\CA}X$. For $n\in\BZ_{>0}$, $\rad_{\CA}^{n}(X,Y)$ denotes
the $\ring$-submodule of $\Hom_{\CA}(X,Y)$ generated by morphisms that
are a composition of $n$ radical morphisms. See \cite{Kelly-radical-of-category}
or \cite[\S 2]{Krause-KS-cats-and-projective-covers} for more details.

The next two propositions together give a combined version of \cite[Lem. 3.8]{AuslanderSmalo-preprojective-modules}
and \cite[Lem. IV.1.9]{AssemSimsonSkowronski-Vol1} valid for any $\ring$-category.
We provide a full proof as the proof of the corresponding result for
module categories is omitted in \cite{AuslanderSmalo-preprojective-modules}.
We also note that the equivalence of \eqref{itm:f non-section} and
\eqref{itm:Im(Hom(f,X) contained in radEndX} in each statement below
has already appeared in the proof of \cite[Thm. 2.4]{AuslanderReiten-Rep-theory-of-Artin-algebras-IV}
in the setting of an additive category with split idempotents.
\begin{prop}
\label{lem:AuSm80 Lem 3.8 for additive cat - EndX local, TFAE for f X to Y (i) f non-section (ii) f radical (iii) Im(Hom(f,X)) contained in radEndX=00003Drad(X,X) (iv) for all Z Im(Hom(f,Z)) contained in rad(X,Z)}Let
$f\colon X\to Y$ be a morphism in a $\ring$-category $\CA$, where $\End_{\CA}X$
is a local ring. Then the following are equivalent: 
\emph{\begin{enumerate}[(i)]
\item\label{itm:f non-section} $f$ \emph{is not a section;}
\item $f\in\rad_\CA(X,Y)$\emph{;}
\item\label{itm:Im(Hom(f,X) contained in radEndX} \emph{$\Image(\Hom_\CA(f,X))\subseteq J(\End_\CA X)=\rad_\CA(X,X)$; and}
\item \emph{for all $Z\in\CA$, the image $\Image(\Hom_\CA(f,Z))$ of the map $\Hom_\CA(f,Z)\colon\Hom_\CA(Y,Z)\to \Hom_\CA(X,Z)$ is contained in $\rad_\CA(X,Z)$.}
\end{enumerate}}
\end{prop}

\begin{proof}
First, assume $f=0$. If $f$ were a section then we would have $1_{X}=gf=0$
for some $g\colon Y\to X$, but this is impossible as $\End_{\CA}X$
is local so $X\neq0$. Thus, $f$ is not a section and (i) holds true.
Furthermore, for any $Z\in\CA$, we have $\Image(\Hom_{\CA}(f,Z))=0$
if $f=0$, and so is contained in $\rad_{\CA}(X,Z)$. That is, (iii)
and (iv) are satisfied. Since $\rad_{\CA}$ is an ideal of $\CA$,
$f=0\in\rad_{\CA}(X,Y)$ and (ii) also holds in this case.

Therefore, we may now assume $f\neq0$. It is clear that (iv) implies
(iii). 

(iii) $\Rightarrow$ (i). If $f$ is a section, then there exists
$g\colon Y\to X$ with $\Hom_{\CA}(f,X)(g)=gf=1_{X}\in\Image(\Hom_{\CA}(f,X))\backslash J(\End_{\CA}X)$,
so $\Image(\Hom_{\CA}(f,X))\nsubseteq J(\End_{\CA}X)$.

(i) $\Rightarrow$ (ii). Suppose $f$ is not a section. Since $\End_{\CA}X$
is local, $\rad_{\CA}(X,X)=J(\End_{\CA}X)$ is the set of all non-left
invertible elements of $\End_{\CA}X$. Let $g\colon Y\to X$ be arbitrary
and consider $gf\colon X\to X$. Notice that $gf$ cannot have a left
inverse because we are assuming $f$ is not a section. Therefore,
$gf\in\rad_{\CA}(X,X)$ and $1_{X}-gf$ is invertible. This is precisely
the requirement for $f$ to be radical.

(ii) $\Rightarrow$ (iv). Suppose $f\in\rad_{\CA}(X,Y)$, and let
$Z\in\CA$ and $g\colon Y\to Z$ be arbitrary. Since $\rad_{\CA}$
is an ideal of $\CA$, we immediately see that $\Hom_{\CA}(f,Z)(g)=gf\in\rad_{\CA}(X,Z)$
and so $\Image(\Hom_{\CA}(f,Z))\subseteq\rad_{\CA}(X,Z)$.
\end{proof}
\begin{prop}
\label{lem:AuSm80 Lem 3.8 for additive cat - EndY local, TFAE for f X to Y (i) f non-retraction (ii) f radical (iii) Im(Hom(Y,f)) contained in radEndY=00003Drad(Y,Y) (iv) for all Z Im(Hom(Z,f)) contained in rad(Z,Y)}Let
$f\colon X\to Y$ be a morphism in a $\ring$-category $\CA$, where $\End_{\CA}Y$
is a local ring. Then the following are equivalent:\emph{\begin{enumerate}[(i)]
\item $f$ \emph{is not a retraction;}
\item $f\in\rad_\CA(X,Y)$\emph{;}
\item $\Image(\Hom_\CA(Y,f))\subseteq J(\End_\CA Y)=\rad_\CA(Y,Y)$\emph{; and}
\item \emph{for all $Z\in\CA$, the image $\Image(\Hom_\CA(Z,f))$ of the map $\Hom_\CA(Z,f)\colon\Hom_\CA(Z,X)\to \Hom_\CA(Z,Y)$ is contained in $\rad_\CA(Z,Y)$.}
\end{enumerate}}
\end{prop}

Immediately from the above two results, we have
\begin{cor}
\label{cor:additive cat f X to Y EndX or End Y local and f non-section non-retraction then f radical}Let
$\CA$ be a $\ring$-category and $f\colon X\to Y$ a morphism in $\CA$,
and suppose $\End_{\CA}X$ is local or $\End_{\CA}Y$ is local. If
$f$ is neither a section nor a retraction, then $f\in\rad_{\CA}(X,Y)$.
\end{cor}

So far we have only studied irreducible morphisms and, as mentioned
earlier, we will also be concerned with (minimal) left/right almost
split morphisms.
\begin{defn}
\label{def:(min) left right almost split}\cite[\S 2]{AuslanderReiten-Rep-theory-of-Artin-algebras-IV}
Let $f\colon X\to Y$ be a morphism in an arbitrary category. We call
$f$ \emph{right almost split} if \begin{enumerate}[(i)]
\item $f$ is not a retraction; and
\item for any non-retraction $u\colon U\to Y$ there exists $\hat{u}\colon U\to X$ such that $f\hat{u}=u$.
\end{enumerate}

If $fg=f$ implies $g$ is an automorphism for any $g\colon X\to X$,
then $f$ is said to be \emph{right minimal}. If $f$ is both right
minimal and right almost split, then $f$ is called \emph{minimal
right almost split}.

Dually, one can define the notions of \emph{left almost split}, \emph{left
minimal} and \emph{minimal left almost split}.
\end{defn}

We recall that in an additive category if $f\colon X\to Y$ is right
almost split (respectively, left almost split), then $Y$ (respectively,
$X$) has local endomorphism ring; see \cite[Lem. 2.3]{AuslanderReiten-Rep-theory-of-Artin-algebras-IV}.

\begin{prop}
\label{prop:left (right) minimal almost split with nonzero target (domain) is irreducible}Let
$\CA$ be an additive category with split idempotents. \emph{\begin{enumerate}[(i)]
\item \emph{If $f\colon X \to Y$ is minimal left almost split and $Y\neq0$, then $f$ is irreducible.}
\item \emph{If $f\colon X \to Y$ is minimal right almost split and $X\neq 0$, then $f$ is irreducible.}
\end{enumerate}}
\end{prop}

\begin{proof}
For (i), notice that $f$ satisfies the criterion in \cite[Thm. 2.4 (b)]{AuslanderReiten-Rep-theory-of-Artin-algebras-IV}.
Statement (ii) is dual. $\phantom{blah}$
\end{proof}
The next proposition is an observation that we may generalise \cite[Prop. 2.18]{BautistaSalorio-irred-morphs-in-bdd-derived-cat}
to a category with split idempotents that is not necessarily Krull-Schmidt,
e.g. the category of all left $R$-modules for a ring $R$, or the
category of all Banach spaces (over $\mathbb{R},$ for example). This
result generalises \cite[Cor. 2.5]{AuslanderReiten-Rep-theory-of-Artin-algebras-IV}
since an irreducible morphism with a domain or codomain that has local
endomorphism ring is radical by Corollary \ref{cor:additive cat f X to Y EndX or End Y local and f non-section non-retraction then f radical}.
We omit the proof as the one given in \cite{BautistaSalorio-irred-morphs-in-bdd-derived-cat}
holds in our generality using \cite[Rem. 7.4]{Buhler-exact-cats}.
See \cite[Prop. 3.2]{Bautista-irred-morphs-and-radical-of-a-cat} also.
\begin{prop}
\label{prop:BaSo11 Prop. 2.18 for add cat with split idems}Let $\CA$
be an additive category with split idempotents. Suppose $f\colon X\to Y$
is a radical irreducible morphism in $\CA$.\emph{\begin{enumerate}[(i)]
\item \emph{If $0\neq g\colon W\to X$ is a section, then $fg\colon W\to Y$ is irreducible.}
\item \emph{If $0\neq h\colon Y\to Z$ is a retraction, then $hf\colon X\to Z$ is irreducible.}
\end{enumerate}}
\end{prop}

The following is a version of \cite[Prop. 2.10]{AuslanderReiten-Rep-theory-of-Artin-algebras-IV}
for the quasi-abelian setting. We will see that the idea behind the
proof is the same, but we have to negotiate around the fact that the
class of kernels (respectively, cokernels) does not necessarily coincide
with the class of monomorphisms (respectively, epimorphisms) in the
category.
\begin{prop}
\label{prop:A quasi-abelian f irred mono indecomp target then any irred map to coker of f is epic - and dual}Suppose
$\CA$ is a quasi-abelian category.\emph{\begin{enumerate}[(i)]
\item \emph{If $f\colon X\to Y$ is an irreducible monomorphism, $Y$ is indecomposable and $v\colon V\to\Cok f$ is any irreducible morphism, then $v$ is epic.}
\item \emph{If $f\colon X\to Y$ is an irreducible epimorphism, $X$ is indecomposable and $u\colon\Ker f\to X$ is any irreducible morphism, then $u$ is monic.}
\end{enumerate}}
\end{prop}

\begin{proof}
We only prove (i) as (ii) is dual. Suppose $f\colon X\to Y$ is an
irreducible monomorphism, and that $Y$ is an indecomposable object.
First, if $C\deff\Cok f=0$ then any morphism $v\colon V\to C$ is
trivially epic as $\CA$ is additive, so we may suppose $C\neq0$.

Consider the (not necessarily short exact) sequence $\begin{tikzcd}[column sep=1.5cm]X \arrow{r}{f}& Y\arrow[two heads]{r}{c\deff\cok f} & C.\end{tikzcd}$
Let $v\colon V\to C$ be an irreducible morphism in $\CA$. By Proposition
\ref{prop:ASS Lem IV.1.7(a) or AR IV Prop 2.7 (a) for quasi-abelian category - criterion for irreducible if f is a kernel}
either there exists $v_{1}\colon V\to Y$ such that $cv_{1}=v$ or
there exists $v_{2}\colon Y\to V$ such that $c=vv_{2}$. In the latter
case, as $c$ is epic, $v$ would also be epic and we would be done.
Thus, suppose no such $v_{2}$ exists. Then there exists $v_{1}\colon V\to Y$
with $cv_{1}=v$. But now $v$ is irreducible, and so either $c$
is a retraction or $v_{1}$ is a section. If $c\colon Y\to C$ is
a retraction, then by Lemma \ref{lem:add cat with split idems, X indecomp and Y neq 0, then f X to Y retraction implies f isomorphism - and dual statement}
we have that $c$ is an isomorphism since $Y$ is indecomposable and
$C\neq0$. However, this implies $f=c^{-1}cf=0$ and in turn yields
$1_{X}=0$, since $f\circ1_{X}=f=0$ and $f$ is monic. Thus, we would
have $X=0$ and $f$ is in fact a section, which contradicts that
$f$ is irreducible. Hence, $c$ cannot be a retraction and so $v_{1}\colon V\to Y$
must be a section.

If $V=0$ then $v\colon0=V\to C$ is a section, which is impossible
as $v$ is assumed to be irreducible. Therefore, $V\neq0$ and hence,
by Lemma \ref{lem:add cat with split idems, X indecomp and Y neq 0, then f X to Y retraction implies f isomorphism - and dual statement}
again, $v_{1}\colon V\to Y$ must be an isomorphism and, in particular,
an epimorphism. Finally, we observe that $v=cv_{1}$ is the composition
of two epimorphisms and hence an epimorphism itself.
\end{proof}
\begin{defn}
\label{def:semi-stable co-kernel}\cite[p. 522]{RichmanWalker-ext-in-preabelian-cats}
Let $\CA$ be an additive category. A kernel (respectively, cokernel)
is called \emph{semi-stable} if every pushout (respectively, pullback)
of it is again a kernel (respectively, a cokernel).
\end{defn}

\begin{example}
Consider a stable exact sequence $\xi\colon X\overset{f}{\to}Y\overset{g}{\to}Z$.
Let $a\colon X\to X'$ be a morphism, and form the sequence $a\xi$
as in Definition \ref{def:cocomplex a xi and xi c for a cocomplex xi}.
Since $\xi$ is stable, the sequence $a\xi$ is short exact and hence
the pushout of $f$ along $a$ is again a kernel. Thus, $f$ is a
semi-stable kernel. Dually, $g$ is a semi-stable cokernel.
\end{example}

\begin{rem}
\label{rem:kernels-semi-stable in right q-ab cat and dual - all seqs stable in q-abelian cat}All
kernels are semi-stable in a right quasi-abelian category and all
cokernels are semi-stable in a left quasi-abelian category. In particular,
all short exact sequences are stable in a quasi-abelian category.
\end{rem}

\begin{lem}
\label{lem:A preabelian morphism 1 v 1 xi to xi is an isomorphism if xi has semi-stable kernel or cokernel}Let
$\CA$ be a preabelian category. Suppose $\xi\colon X\overset{f}{\to}Y\overset{g}{\to}Z$
is short exact. If $f$ or $g$ is semi-stable, then any morphism
$(1_{X},v,1_{Z})\colon\xi\to\xi$ is an isomorphism.
\end{lem}

\begin{proof}
This follows from \cite[Thm. 6]{RichmanWalker-ext-in-preabelian-cats},
noting that the dual of this result by Richman and Walker holds when
the morphism $(1_{X},v,1_{Z})$ of short exact sequences is an endomorphism. \phantom{spacefiller}
\end{proof}
Our main theorem in \S\ref{sec:Auslander-Reiten-Theory-in-Krull-Schmidt-Categories}
generalises \cite[Thm. IV.1.13]{AssemSimsonSkowronski-Vol1}, and part
of the proof uses tools to detect when an endomorphism $(u,v,w)$ of
a short exact sequence is in fact an \emph{isomorphism}, i.e. $u,v,w$
are all isomorphisms. We present generalisations of these tools now,
and we will see the work of \S\ref{subsec:Ext-in-preabelian-cat}
used below. We will assume for simplicity that a preabelian
category $\CA$ is skeletally small whenever our proofs require the
use of an extension group. However, we only really need that the first
extension group is a set in the relevant arguments (see Remark \ref{rem:need Ext(Z X) to be a set to be a group}). 
\begin{defn}
\label{def:Hom-finite}\cite[p. 547]{Krause-KS-cats-and-projective-covers}
An additive category $\CA$ is called \emph{$\Hom$-finite} if $\CA$
is a $\ring$-category, for some commutative ring $\ring$, and $\Hom_{\CA}(X,Y)$
is a finite length $\ring$-module for any $X,Y\in\CA$.
\end{defn}

\begin{prop}
\label{prop:ASS Lem IV.1.12 in Hom-finite cat endo u v w of non-split short exact seq X f Y g Z with EndX (EndZ) local and w (u) iso implies u (w) iso and if f or g semi-stable then v iso too}Let
$\CA$ be a $\Hom$-finite category. Suppose we have a commutative
diagram$$\begin{tikzcd}X\arrow{r}{f}\arrow{d}{u}&Y\arrow{r}{g}\arrow{d}{v}&Z\arrow{d}{w} \\ X\arrow{r}[swap]{f}&Y\arrow{r}[swap]{g}&Z\end{tikzcd}$$in
$\CA$ with non-split short exact rows. If $\End_{\CA}X$ (respectively,
$\End_{\CA}Z$) is local and $w$ (respectively, $u$) is an automorphism,
then $u$ (respectively, $w$) is an automorphism.

Further, if $\CA$ is also preabelian and if $f$ or $g$ is semi-stable,
then $v$ is also an automorphism in this case.
\end{prop}

\begin{proof}
Suppose that $\End_{\CA}X$ is local and $w$ is an automorphism of
$Z$. Showing that $u$ is an automorphism in this case is the same
as in \cite[Lem. IV.1.12]{AssemSimsonSkowronski-Vol1}.

Now assume $\CA$ is also preabelian, and that $f$ is a semi-stable
kernel or $g$ is a semi-stable cokernel. Consider the commutative
diagram $$\begin{tikzcd}X\arrow{r}{fu^{-1}}\arrow[equal]{d}&Y\arrow{r}{wg}\arrow{d}{v}&Z\arrow[equal]{d} \\ X\arrow{r}[swap]{f}&Y\arrow{r}[swap]{g}&Z\end{tikzcd}$$that
has short exact rows. Then $v$ is an automorphism by Lemma \ref{lem:A preabelian morphism 1 v 1 xi to xi is an isomorphism if xi has semi-stable kernel or cokernel}.
\end{proof}
The following corollary is a generalisation of \cite[Lem. IV.1.12]{AssemSimsonSkowronski-Vol1}
to a quasi-abelian $\Hom$-finite setting, and it follows quickly
from Proposition \ref{prop:ASS Lem IV.1.12 in Hom-finite cat endo u v w of non-split short exact seq X f Y g Z with EndX (EndZ) local and w (u) iso implies u (w) iso and if f or g semi-stable then v iso too}
in light of Remark \ref{rem:kernels-semi-stable in right q-ab cat and dual - all seqs stable in q-abelian cat}.
\begin{cor}
\label{cor:ASS Lem IV.1.12 in Hom-finite cat endo u v w of non-split short exact seq X f Y g Z with EndX (EndZ) local and w (u) iso implies u (w) iso and if cat left or right quasi-abelian then v iso too}Let
$\CA$ be a $\Hom$-finite category, which is left or right quasi-abelian.
Suppose we have a commutative diagram$$\begin{tikzcd}X\arrow{r}{f}\arrow{d}{u}&Y\arrow{r}{g}\arrow{d}{v}&Z\arrow{d}{w} \\ X\arrow{r}[swap]{f}&Y\arrow{r}[swap]{g}&Z\end{tikzcd}$$in
$\CA$ with non-split short exact rows. If $\End_{\CA}X$ (respectively,
$\End_{\CA}Z$) is local and $w$ (respectively, $u$) is an automorphism,
then $u$ (respectively, $w$) is an automorphism. Furthermore, $v$
is also an automorphism.
\end{cor}

The next proposition is a generalisation of \cite[Lem. 2.13]{AuslanderReiten-Rep-theory-of-Artin-algebras-IV}
for preabelian categories. However, note that we need to assume the
short exact sequence in question is stable.
\begin{prop}
\label{prop:in preabelian cat endo u v w of non-split stable short exact seq X to Y to Z with EndX (EndZ) local and w (u) iso implies u (w) iso and v iso too}Let
$\CA$ be a skeletally small, preabelian category. Suppose we have
a commutative diagram$$\begin{tikzcd}X\arrow{r}{f}\arrow{d}{u}&Y\arrow{r}{g}\arrow{d}{v}&Z\arrow{d}{w} \\ X\arrow{r}[swap]{f}&Y\arrow{r}[swap]{g}&Z\end{tikzcd}$$in
$\CA$, where $\xi\colon X\overset{f}{\to}Y\overset{g}{\to}Z$ is
a non-split stable exact sequence. If $\End_{\CA}X$ (respectively,
$\End_{\CA}Z$) is local and $w$ (respectively, $u$) is an automorphism,
then $u$ (respectively, $w$) and hence $v$ are automorphisms.
\end{prop}

\begin{proof}
The proof is that of \cite{AuslanderReiten-Rep-theory-of-Artin-algebras-IV}
with the following adjustments. In order to show $u$ is an automorphism,
one needs that $u\xi\iso\xi1_{Z}=\xi$ and that $\Ext_{\CA}^{1}(Z,X)$
is a left $\End_{\CA}X$-module, which follow from \cite[Cor. 7]{RichmanWalker-ext-in-preabelian-cats}
and Theorem \ref{thm:Ext^1 is a left-right EndX-EndZ bimodule in preabelian category},
respectively. Lastly, an application of Lemma \ref{lem:A preabelian morphism 1 v 1 xi to xi is an isomorphism if xi has semi-stable kernel or cokernel}
yields that $v$ is also an automorphism.
\end{proof}
Since all short exact sequences are stable in a quasi-abelian category
(see Remark \ref{rem:kernels-semi-stable in right q-ab cat and dual - all seqs stable in q-abelian cat}),
we obtain a direct generalisation of \cite[Lem. 2.13]{AuslanderReiten-Rep-theory-of-Artin-algebras-IV}
as follows.
\begin{cor}
\label{cor:in quasi-abelian category endo u v w of non-split short exact seq X to Y to Z with EndX (EndZ) local and w (u) iso implies u (w) iso and v iso too}Let
$\CA$ be a skeletally small, quasi-abelian category. Suppose we have
a commutative diagram$$\begin{tikzcd}X\arrow{r}{f}\arrow{d}{u}&Y\arrow{r}{g}\arrow{d}{v}&Z\arrow{d}{w} \\ X\arrow{r}[swap]{f}&Y\arrow{r}[swap]{g}&Z\end{tikzcd}$$in
$\CA$, where $X\overset{f}{\to}Y\overset{g}{\to}Z$ is a non-split
short exact sequence. If $\End_{\CA}X$ (respectively, $\End_{\CA}Z$)
is local and $w$ (respectively, $u$) is an automorphism, then $u$
(respectively, $w$) and hence $v$ are automorphisms.
\end{cor}

The following definition has been given by Auslander for abelian
categories (see \cite[p. 292]{Auslander-Rep-theory-of-Artin-algebras-I-II}),
but we may make the same definition for additive categories and are
able to derive some of the same consequences.
\begin{defn}
\label{def:minimal element of F(X) where X in A additive and F additive covariant functor from A to cat of abelian groups, and dual definition for contravariant functor}\cite[p. 292]{Auslander-Rep-theory-of-Artin-algebras-I-II}
Let $\CA$ be an additive category with an object $X$. Suppose $\SF\colon\CA\to\Ab$
is a covariant additive functor to the category $\Ab$ of all abelian
groups. An element $x\in\SF(X)$ is said to be \emph{minimal} if $x\neq0$
(where $0$ is the identity element of the abelian group $\SF(X)$),
and if for all proper epimorphisms $f\colon X\to Y$ in $\CA$, we
have that $\SF(f)\colon\SF(X)\to\SF(Y)$ satisfies $\SF(f)(x)=0$.

A definition of minimal can be made for a contravariant functor $\SG\colon\CA\to\Ab$
by considering $\SG$ as a covariant functor $\CA^{\op}\to\Ab$. 
\end{defn}

An immediate result is a version of \cite[p. 292, Lem. 3.2 (a)]{Auslander-Rep-theory-of-Artin-algebras-I-II}
for additive categories:
\begin{prop}
\label{prop:additive cat A, F functor A to Ab, F(X) has minimal element implies X indecomposable}Let
$\CA$ be an additive category with an object $X$. Suppose $\SF\colon\CA\to\Ab$
is a covariant additive functor. If $\SF(X)$ has a minimal element,
then $X$ is indecomposable in $\CA$.
\end{prop}

\begin{proof}
Assume $x\in\SF(X)$ is minimal, and that $X=X_{1}\oplus X_{2}$ with
$X_{1},X_{2}$ both non-zero. Let $\iota_{i}\colon X_{i}\into X$
and $\pi_{i}\colon X\onto X_{i}$ be the canonical inclusion and projection morphisms, respectively, for $i=1,2$. Note that $\pi_{i}$
is a proper epimorphism for $i=1,2$ since $X_{1}$ and $X_{2}$ are
non-zero. Therefore, $\SF(\pi_{i})(x)=0$ for $i=1,2$ as $x$ is
minimal. However, this implies 
\[
\begin{array}{rcll}
x & = & 1_{\SF(X)}(x)=\SF(1_{X})(x) & \text{as }\SF\text{ is a functor}\\
 & = & \SF(\iota_{1}\pi_{1}+\iota_{2}\pi_{2})(x) & \text{as }X=X_{1}\oplus X_{2}\\
 & = & \SF(\iota_{1}\pi_{1})(x)+\SF(\iota_{2}\pi_{2})(x) & \text{as }\SF\text{ is additive}\\
 & = & \SF(\iota_{1})(\SF(\pi_{1})(x))+\SF(\iota_{2})(\SF(\pi_{2})(x)) & \text{as }\SF\text{ is covariant}\\
 & = & 0 & \text{since }\SF(\pi_{i})(x)=0.
\end{array}
\]
This is a contradiction because $x\neq0$ since it is minimal. Hence,
$X$ must be indecomposable.
\end{proof}
The next proposition generalises \cite[Prop. 2.6 (b)]{AuslanderReiten-Rep-theory-of-Artin-algebras-IV}
to a semi-abelian setting. The strategy in the proof is the same,
but we need a technical result from \cite{Rump-almost-abelian-cats}
in order to work in a category with less structure. Note that if $\CA$ skeletally small, then 
$\Ext_{\CA}^{1}(-,-)$ is an additive bifunctor (see \cite[\S4]{RichmanWalker-ext-in-preabelian-cats}, or \cite[p. 267]{Cooper-stable-seqs-in-preabelian-cats}).
\begin{prop}
\label{prop:A semi-abelian xi X f Y g Z short exact - f irred implies xi min and Z indecomp - and dually for g}Let
$\CA$ be a skeletally small, semi-abelian category with objects $X,Z$.
Consider the covariant additive functor $\Ext_{\CA}^{1}(Z,-)\colon\CA\to\Ab$
and the contravariant additive functor $\Ext_{\CA}^{1}(-,X)\colon\CA\to\Ab$.
Suppose $\xi\colon X\overset{f}{\to}Y\overset{g}{\to}Z$ is an element
of $\Ext_{\CA}^{1}(Z,X)$.\emph{\begin{enumerate}[(i)]
\item \emph{If $f$ is irreducible, then $\xi\in\Ext_{\CA}^{1}(-,X)(Z)$ is minimal, and hence $Z$ is indecomposable.}
\item \emph{If $g$ is irreducible, then $\xi\in\Ext_{\CA}^{1}(Z,-)(X)$ is minimal, and hence $X$ is indecomposable.}
\end{enumerate}}
\end{prop}

\begin{proof}
We prove (ii); the proof for (i) is similar. Suppose $g$ is irreducible
in the stable exact sequence $\xi\colon X\overset{f}{\to}Y\overset{g}{\to}Z$.
Since $g$ is not a retraction, by the Splitting Lemma (Proposition
\ref{prop:Splitting-Lemma in additive category}) we know $\xi$ is
not split and hence $\xi\neq0$ in $\Ext_{\CA}^{1}(Z,X)$. Suppose
$a\colon X\to X_{1}$ is a proper epimorphism. We will show that $\Ext_{\CA}^{1}(Z,a)(\xi)=a\xi=0$,
i.e. the short exact sequence $a\xi$ is split. By definition, $a\xi$
comes with some commutative diagram$$\begin{tikzcd} \xi \colon & X \arrow{r}{f}\arrow{d}{a}& Y\arrow{r}{g} \arrow{d}{b}& Z \arrow[equals]{d} \\
a\xi\colon & X_1 \arrow{r}{f_1}& Y_1\arrow{r}{g_1}\PO{ul} & Z\end{tikzcd}$$where the left square is a pushout square. Thus, $g=g_{1}b$ and so
$g_{1}$ is a retraction or $b$ is a section as $g$ is assumed to
be irreducible.

Assume, for contradiction, that $b$ is a section. Then there exists
$r\colon Y_{1}\to Y$ such that $rb=1_{Y}$. This yields $(rf_{1})a=rbf=f=\ker g$.
As $\CA$ is (right) semi-abelian, we have that $a$ is also a kernel
by \cite[Prop. 2]{Rump-almost-abelian-cats}. Therefore,
$a$ is an epic kernel and hence an isomorphism by Lemma \ref{lem:epic kernel or monic cokernel is an isomorphism},
which contradicts that $a$ is a proper epimorphism. Hence, $b$ cannot
be a section.

Thus, $g_{1}$ must be a retraction, whence $a\xi\colon X_{1}\overset{f_{1}}{\longrightarrow}Y_{1}\overset{g_{1}}{\longrightarrow}Z$
is split (Proposition \ref{prop:Splitting-Lemma in additive category}
again) and $a\xi=0$ in $\Ext_{\CA}^{1}(Z,-)(X_{1})$. Since $a\colon X\to X_{1}$
was an arbitrary proper epimorphism, we see that $\xi$ is a minimal
element of $\Ext_{\CA}^{1}(Z,-)(X)$ and that $X$ is indecomposable
by Proposition \ref{prop:additive cat A, F functor A to Ab, F(X) has minimal element implies X indecomposable}.
\end{proof}
We also observe that \cite[Prop. 2.11]{AuslanderReiten-Rep-theory-of-Artin-algebras-IV}
remains valid in a more general situation:
\begin{prop}
\label{prop:A preabelian, f X to Y irred, A arb object then (i) Hom(Y,A) is 0 implies 0 to Ext(A,X) to Ext(A,Y) exact (ii) Hom(A,X) is 0 implies 0 to Ext(Y,A) to Ext(X,A)}Let
$\CA$ be a skeletally small, preabelian category. Suppose $f\colon X\to Y$
is an irreducible morphism in $\CA$ and let $Z$ be an object of
$\CA$.\emph{\begin{enumerate}[(i)]
\item \emph{If $\Hom_{\CA}(Y,Z)=0$, then $\begin{tikzcd}[column sep=1.6cm, ampersand replacement = \&] 0 \arrow{r} \& \Ext_{\CA}^{1}(Z,X) \arrow{r}{\Ext_{\CA}^{1}(Z,f)} \& \Ext_{\CA}^{1}(Z,Y) \end{tikzcd}$ is exact.}
\item \emph{If $\Hom_{\CA}(Z,X)=0$, then $\begin{tikzcd}[column sep=1.6cm, ampersand replacement = \&] 0 \arrow{r} \& \Ext_{\CA}^{1}(Y,Z) \arrow{r}{\Ext_{\CA}^{1}(f,Z)} \& \Ext_{\CA}^{1}(X,Z) \end{tikzcd}$ is exact.}
\end{enumerate}}
\end{prop}

\begin{proof}
This is an arrow-theoretic translation of the proof from \cite{AuslanderReiten-Rep-theory-of-Artin-algebras-IV}.
\end{proof}
The next result is an analogue of \cite[Prop. 2.8]{AuslanderReiten-Rep-theory-of-Artin-algebras-IV}
for which we need the theory of subobjects in an abelian category.
Recall that two monomorphisms $i_{1}\colon X_{1}\to X$ and $i_{2}\colon X_{2}\to X$
in an abelian category are said to be \emph{equivalent} if there is
an isomorphism $f\colon X_{1}\overset{\iso}{\longrightarrow}X_{2}$
such that $i_{1}=i_{2}\circ f$. Then a \emph{subobject} of an object
$X$ in an abelian category is an equivalence class of monomorphisms
into $X$. Furthermore, if $i\colon V\to X$ and $j\colon W\to X$
are representatives of subobjects of $X$, then we write $V\subseteq W$
if there exists a morphism $g\colon V\to W$ such that $i=j\circ g$.
See \cite{Krause-KS-cats-and-projective-covers} or \cite{Freyd-abelian-cats}
for more details.
\begin{prop}
\label{prop:AR IV Prop. 2.8 f X to Y, g Y to Z then (i) f irred kernel iff for all M subfunctor of (-,Z) M contains or contained in Im(-,g) and (ii) dual statement}Let
$\CA$ be a skeletally small, quasi-abelian category and suppose $X\overset{f}{\to}Y\overset{g}{\to}Z$
is a non-split short exact sequence in $\CA$.\emph{\begin{enumerate}[(i)]
\item \emph{The morphism} $f$ \emph{is irreducible if and only if, for any subobject} $\SF$ \emph{of} $\Hom_{\CA}(-,Z)$\emph{, we have either} $\SF$ \emph{contains or is contained in} $\Image(\Hom_{\CA}(-,g))$\emph{, the image of the natural transformation} $\Hom_{\CA}(-,g)\colon \Hom_{\CA}(-,Y)\to \Hom_{\CA}(-,Z)$\emph{.}
\item \emph{The morphism} $g$ \emph{is irreducible if and only if, for any subobject} $\SF$ \emph{of} $\Hom_{\CA}(X,-)$\emph{, we have either} $\SF$ \emph{contains or is contained in} $\Image(\Hom_{\CA}(f,-))$\emph{, the image of the natural transformation} $\Hom_{\CA}(f,-)\colon \Hom_{\CA}(Y,-)\to \Hom_{\CA}(X,-)$\emph{.}
\end{enumerate}}
\end{prop}

\begin{proof}
Note that the category of functors $\CA\to\Ab$ is abelian as $\CA$
is skeletally small (see \cite[Thm. 10.1.3]{Prest-purity-spectra-localisation}).
The proof is identical to that for \cite[Prop. 2.8]{AuslanderReiten-Rep-theory-of-Artin-algebras-IV},
noting that we may use \cite[Prop. 10.1.13]{Prest-purity-spectra-localisation},
and Propositions \ref{prop:ASS Lem IV.1.7(a) or AR IV Prop 2.7 (a) for quasi-abelian category - criterion for irreducible if f is a kernel}
and \ref{prop:ASS Lem IV.1.7(b) or AR IV Prop 2.7 (b) for quasi-abelian category - criterion for irreducible if f is a cokernel}.
\end{proof}
Proposition \ref{prop:AR IV Cor 2.9 (a) for quasi-abelian category}
and its dual, Proposition \ref{prop:AR IV Cor 2.9 (b) for quasi-abelian category},
below give analogues of one direction of parts (a) and (b) of \cite[Cor. 2.9]{AuslanderReiten-Rep-theory-of-Artin-algebras-IV}
for quasi-abelian categories. There is an obvious method to prove
(ii) in light of Proposition \ref{prop:ASS Lem IV.1.7(a) or AR IV Prop 2.7 (a) for quasi-abelian category - criterion for irreducible if f is a kernel},
but the details are omitted in \cite{AuslanderReiten-Rep-theory-of-Artin-algebras-IV}
so we include them here for completeness.
\begin{prop}
\label{prop:AR IV Cor 2.9 (a) for quasi-abelian category}Let $\CA$
be a quasi-abelian category and suppose $\xi\colon X\overset{f}{\to}Y\overset{g}{\to}Z$
is a non-split short exact sequence in $\CA$. Suppose $f$ is irreducible.
Then the following statements hold.\emph{\begin{enumerate}[(i)]
\item \emph{For any proper subobject} $\iota \colon Y'\into Y$ \emph{such that} $\Image f\subseteq Y'$ \emph{given by a monomorphism} $s\colon \Image f \into Y'$\emph{, we have that} $s$ \emph{is a section.}
\item \emph{For any short exact sequence} $\xi '\colon A\overset{a}{\to} B\overset{b}{\to} Z$\emph{, either there exists} $j\colon A\to X$ \emph{such that} $j\xi '=\xi$ \emph{or there exists} $i\colon X\to A$ \emph{such that} $\xi i=\xi '$\emph{.}
\end{enumerate}}
\end{prop}

\begin{proof}
Suppose $f$ is irreducible. To show (i) holds, we assume $\iota\colon Y'\to Y$
is a proper monomorphism and $s\colon\Image f\to Y'$ is a monomorphism
such that $\im f=\iota s$. Since $\CA$ is preabelian and $\xi$
is short exact, we have that $f=\ker g=\ker(\cok f)=\image f$ by
Lemma \ref{lem:ker is ker of its cok - cok is cok of its ker}. Therefore,
$f=\iota s$ and hence either $\iota$ is a retraction or $s$ is
a section. If $\iota$ is a retraction then it would be a monic retraction,
and hence an isomorphism by \cite[Thm. I.1.5]{McLarty-elementary-categories-elementary-toposes}.
But this contradicts our assumption on $\iota$, so we must have that
$s$ is a section.

For (ii), if $A\overset{a}{\to}B\overset{b}{\to}Z$ is a short exact
sequence, then by Proposition \ref{prop:ASS Lem IV.1.7(a) or AR IV Prop 2.7 (a) for quasi-abelian category - criterion for irreducible if f is a kernel}
either there exists $v_{1}\colon B\to Y$ with $gv_{1}=b$ or there
exists $v_{2}\colon Y\to B$ with $g=bv_{2}$. This will yield one
of the two morphisms of short exact sequences indicated in the following
diagram.$$\begin{tikzcd}[row sep=1.2cm, column sep=2cm]A \arrow{r}{a}\arrow[dotted, bend right]{d}[swap]{\exists ! i_1}\arrow[mysymbol]{dr}[description]{\text{I}}& B \arrow{r}{b}\arrow[bend right]{d}[swap]{v_1}& Z\arrow[equals]{d}{1_Z} \\
X\arrow{r}{f}\arrow[dotted, bend right]{u}[swap]{\exists ! i_2} & Y\arrow{r}{g} \arrow[bend right]{u}[swap]{v_2}& Z
\end{tikzcd}$$Therefore, we need only show that the left square \rom{1} is a pushout
in either case. However, this follows immediately from the dual of
\cite[Prop. 5]{Rump-almost-abelian-cats} since $a,f$
are kernels and $1_{Z}$ is an isomorphism.
\end{proof}
\begin{prop}
\label{prop:AR IV Cor 2.9 (b) for quasi-abelian category}Let $\CA$
be a quasi-abelian category and suppose $\xi\colon X\overset{f}{\to}Y\overset{g}{\to}Z$
is a non-split short exact sequence in $\CA$. Suppose $g$ is irreducible.
Then the following statements hold.\emph{\begin{enumerate}[(i)]
\item \emph{For any non-zero subobject} $X' \overset{\iota}{\into} X$\emph{, the induced morphism} $r\colon Y/X'=\Cok(f\iota) \to Z$ \emph{is a retraction.}
\item \emph{For any short exact sequence} $\xi '\colon X\overset{b}{\to} B\overset{c}{\to} C$\emph{, either there exists} $j\colon Z\to C$ \emph{such that} $\xi 'j=\xi$ \emph{or there exists} $i\colon C\to Z$ \emph{such that} $i\xi=\xi '$\emph{.}
\end{enumerate}}
\end{prop}

%%%%%%%%%%%%%%%%%%%%%%%%%%%%%%%%%%%%%%%%%%%%
\section{\label{sec:Auslander-Reiten-Theory-in-Krull-Schmidt-Categories}Auslander-Reiten
theory in Krull-Schmidt categories}

Let $\CA$ denote a $\ring$-category, for some commutative ring $\ring$,
and let $\CI$ denote a (two-sided) ideal of $\CA$. For a morphism
$f\colon X\to Y$ in $\CA$, we will denote by $\overline{f}$ the
morphism $f+\CI(X,Y)$ in the additive quotient $\ring$-category $\CA/\CI$.
For most of this section we will study Krull-Schmidt categories that are not necessarily $\Hom$-finite. There are very interesting examples of $\Hom$-infinite generalised cluster categories (see \cite{Amiot-cluster-categories-for-algebras-of-global-dimension-2-and-quivers-with-potential}, \cite{Plamondon-cluster-algebras-cluster-categories-with-infinite-dimensional-morphism-spaces}, \cite{KellerYang-derived-equivalences-from-mutations-of-quivers-with-potential}) coming from quivers with potential. In these examples, a certain Krull-Schmidt category has been used in \cite{Plamondon-cluster-characters-for-cluster-categories-with-infinite-dimensional-morphism-spaces} to show the existence of cluster characters for $\Hom$-infinite cluster categories. We provide an example of this kind at the end of this section (see Example \ref{exa:Markov-quiver}).

Before we begin our study of Auslander-Reiten theory in Krull-Schmidt categories, we present a series of lemmas
inspired by \cite[Lem. 1.1]{AuslanderReiten-Rep-theory-of-Artin-algebras-V}.
The proofs are omitted since they are easy generalisations of those
in \cite{AuslanderReiten-Rep-theory-of-Artin-algebras-V}.
\begin{lem}
\label{lem:f in ideal I and 1_X,1_Y notin I then f not isomorphism}Suppose
$X,Y\in\CA$ and $f\in\CI(X,Y)$. If $1_{X}\notin\CI(X,X)$ or $1_{Y}\notin\CI(Y,Y)$,
then $f\colon X\to Y$ is not an isomorphism.
\end{lem}

\begin{lem}
\label{lem:1_Xi notin in ideal I, then endos of X factoring through I are radical}Suppose
$X=\bigoplus_{i=1}^{n}X_{i}$ in $\CA$, with $\End_{\CA}X_{i}$ local
and $1_{X_{i}}\notin\CI(X_{i},X_{i})$ for each $i=1,\ldots,n$. For
an endomorphism $f\colon X\to X$, if $f\in\CI(X,X)$ then $f\in\rad_{\CA}(X,X)$.
\end{lem}

\begin{lem}
\label{lem:X cap I is zero then f X to Y in A is sect iff f mod I is sect, and Y cap I is zero then f is ret iff f mod I ret}Suppose
$X=\bigoplus_{i=1}^{n}X_{i}$ and $Y=\bigoplus_{j=1}^{m}Y_{j}$ in
$\CA$, with $\End_{\CA}X_{i}$ and $\End_{\CA}Y_{j}$ local for all
$i,j$. Let $f\colon X\to Y$ be a morphism in $\CA$.\emph{\begin{enumerate}[(i)]
\item \emph{If} $1_{X_{i}}\notin \CI(X_i,X_i) \;\forall 1\leq i\leq n$\emph{, then} $f$ \emph{is a section in} $\CA \iff \overline{f}$ \emph{is a section in} $\CA/\CI$\emph{.}
\item \emph{If} $1_{Y_{j}}\notin \CI(Y_j,Y_j) \;\forall 1\leq j\leq m$\emph{, then} $f$ \emph{is a retraction in} $\CA \iff \overline{f}$ \emph{is a retraction in} $\CA/\CI$\emph{.}
\item \emph{If} $1_{X_{i}}\notin \CI(X_i,X_i)$ \emph{and} $1_{Y_{j}}\notin \CI(Y_j,Y_j)$ \emph{for all} $i,j$\emph{, then} $f$ \emph{is an isomorphism in} $\CA \iff \overline{f}$ \emph{is an isomorphism in} $\CA/\CI$\emph{.}
\end{enumerate}}
\end{lem}

The forward direction of the next lemma can be found in \cite{LiuS-AR-theory-in-KS-cat}
just above \cite[Def. 1.6]{LiuS-AR-theory-in-KS-cat},
but it is a short argument so we include it here for completeness.
\begin{lem}
\label{lem:Let I be an ideal of additive A, X cap I is 0 implies EndX in A local iff EndX in A mod I local}Suppose
$X=\bigoplus_{i=1}^{n}X_{i}$ in $\CA$, with $\End_{\CA}X_{i}$ local
and $1_{X_{i}}\notin\CI(X_{i},X_{i})$ for each $i=1,\ldots,n$. Then
$\End_{\CA}X$ is local if and only if $\End_{\CA/\CI}X$ is local.
\end{lem}

\begin{proof}
$(\Rightarrow)$ If $\End_{\CA}X$ is local, then $\CI(X,X)$ is contained
in the Jacobson radical $J(\End_{\CA}X)$, which is the unique maximal ideal of $\End_{\CA}X$, since $1_{X}\notin\CI(X,X)$. Then $\End_{\CA/\CI}X$
has unique maximal ideal $J(\End_{\CA}X)/\CI(X,X)$.

$(\Leftarrow)$ Conversely, assume $\End_{\CA/\CI}X$ is local and
let $u\colon X\to X$ be a non-unit in $\End_{\CA}X$. We will show
that $1_{X}-u$ is a unit in $\End_{\CA}X$. If $\overline{u}$ is
a unit in $\End_{\CA/\CI}X$ then $\overline{1_{X}}=\overline{u}\overline{v}$
for some $v\colon X\to X$. Then $1_{X}-uv\in\CI(X,X)$, so $1_{X}-uv$
is radical by Lemma \ref{lem:1_Xi notin in ideal I, then endos of X factoring through I are radical}.
Then $uv=1_{X}-(1_{X}-uv)$ is a unit, so that $u$ has a right inverse.
A similar argument also shows $u$ has a left inverse and so $u$
is a unit, contrary to our assumption on $u$. Thus, $\overline{u}$
is not invertible and, as $\End_{\CA/\CI}X$ is local, $\overline{u}$
must be radical. Therefore, $\overline{1_{X}}-\overline{u}$ is a
unit in $\End_{\CA/\CI}X$. So for some $w\colon X\to X$ we have
$\overline{w}(\overline{1_{X}}-\overline{u})=\overline{1_{X}}$. This
shows that $1_{X}-w(1_{X}-u)\in\CI(X,X)$ must be radical using Lemma
\ref{lem:1_Xi notin in ideal I, then endos of X factoring through I are radical}
as before, so $w(1_{X}-u)=1_{X}-(1_{X}-w(1_{X}-u))$ is invertible
and $1_{X}-u$ has a left inverse. Again, a similar argument shows
$1_{X}-u$ has a right inverse, and hence a two-sided inverse.
\end{proof}

\begin{defn}
\label{def:weak-co-kernel}Suppose $f\colon X\to Y$ is a morphism
in an additive category $\CA$. A \emph{weak kernel} of $f$ is a morphism
$w\colon W\to X$ such that $f\circ w=0$,
with the following property: for every morphism $g\colon V\to X$
with $fg=0$, there exists $\widehat{g}\colon V\to W$ such
that $w\widehat{g}=g$. A \emph{weak cokernel} is defined
dually. We call a sequence $X\overset{f}{\to}Y\overset{g}{\to}Z$
\emph{short weak exact} if $f$ is a weak kernel of $g$ and $g$
is a weak cokernel of $f$.
\end{defn}

It is easy to show that a morphism is a pseudo-(co)kernel, in the
sense of \cite{LiuS-AR-theory-in-KS-cat}, if and only
if it is a weak (co)kernel. We adopt the terminology `weak' as it
seems more widely used.
\begin{defn}
\label{def:AR-sequence in an additive category}We call a sequence
$X\overset{f}{\to}Y\overset{g}{\to}Z$ in $\CA$ an \emph{Auslander-Reiten
sequence (in an additive category)} if the following conditions are
satisfied. \begin{enumerate}[(i)]
\item The sequence is short weak exact.
\item The morphism $f$ is minimal left almost split.
\item The morphism $g$ is minimal right almost split.
\end{enumerate}
\end{defn}

In an almost identical way, Liu defines an Auslander-Reiten sequence
for a $\Hom$-finite, Krull-Schmidt category in \cite{LiuS-AR-theory-in-KS-cat}.
However, we do not impose the condition that the middle term be non-zero,
because Auslander-Reiten sequences of the form $X\to0\to Z$ do appear,
for example, in the bounded derived category $\der^{b}(\lmod{k\BA_{1}})$
of the path algebra $k\BA_{1}$, where $k$ is an algebraically closed
field, and $\BA_{1}$ is the quiver with one vertex and no arrows.
As we will see now, the results of \cite[\S 1]{LiuS-AR-theory-in-KS-cat}
can be generalised to the not necessarily $\Hom$-finite setting.
First, we note that \cite[Lem. 1.1]{LiuS-AR-theory-in-KS-cat}
is still valid for an arbitrary additive category.

Next, we give a more general version of \cite[Prop. 1.5]{LiuS-AR-theory-in-KS-cat}:
\begin{prop}
\label{prop:A preabelian, AR sequence (weak exact) is in fact exact}Suppose
$\CA$ is a preabelian category, and let $\xi\colon X\overset{f}{\to}Y\overset{g}{\to}Z$
be an Auslander-Reiten sequence in $\CA$ with $Y\neq0$. Then $\xi$
is short exact.
\end{prop}

Recall that an additive category $\CA$ is \emph{Krull-Schmidt} if, for any object $X$ of $\CA$, there exists a finite direct
sum decomposition $X=X_{1}\oplus\cdots\oplus X_{n}$ where $\End_{\CA}(X_{i})$
is a local ring for $i=1,\ldots,n$ (see \cite[p. 544]{Krause-KS-cats-and-projective-covers}). 

Throughout the remainder of this section, we further assume that $\CA$
is a Krull-Schmidt category unless otherwise stated. In particular,
$\CA$ has split idempotents (see Remark \ref{rem:preabelian cats have split idempotents}), and an object $X\in\CA$ is indecomposable if and only if $\End_{\CA}X$
is local. We still assume $\CI$ is an ideal of $\CA$.

The following lemma is a generalisation of \cite[Lem. 1.2]{LiuS-AR-theory-in-KS-cat}
that will be needed to prove a uniqueness result about Auslander-Reiten
sequences in a Krull-Schmidt category (see Theorem \ref{thm:Liu thm 1.4}).
Part of the proof in \cite{LiuS-AR-theory-in-KS-cat}
uses heavily that the category is $\Hom$-finite, so the corresponding
part of the proof below is quite different in nature.
\begin{lem}
\label{lem:Liu lem 1.2}Suppose $$\begin{tikzcd} Y \arrow{r}{f}\arrow{d}[swap]{u}& Z\arrow{d}{v} \\ Y\arrow{r}{g} & Z \end{tikzcd}$$is
a commutative diagram in $\CA$, with $f,g$ both non-zero.\emph{\begin{enumerate}[(i)]
\item \emph{If} $f,g$ \emph{are minimal right almost split, then} $u\in\Aut_{\CA}Y$ $\iff$ $v\in\Aut_{\CA}Z$\emph{.}
\item \emph{If} $f,g$ \emph{are minimal left almost split, then} $u\in\Aut_{\CA}Y$ $\iff$ $v\in\Aut_{\CA}Z$\emph{.}
\end{enumerate}}
\end{lem}

\begin{proof}
We prove only (i) as the proof for (ii) is dual. Assume $f,g$ are
non-zero, minimal right almost split morphisms with $vf=gu$. Note
that the argument in \cite{LiuS-AR-theory-in-KS-cat}
that $u$ is an automorphism of $Y$ whenever $v$ is an automorphism
of $Z$ works here as well, so we only show the converse. We observe
for later use that $Y,Z$ are both non-zero since there exists a non-zero
morphism between them.

Therefore, suppose $u\in\Aut_{\CA}Y$ with inverse $u^{-1}$. Since
$f$ is right almost split, we have that $\End_{\CA}Z$ is local,
so $Z$ is indecomposable. Assume, for contradiction, that $v$ is
not a retraction. Then $v$ factors through the right almost split
morphism $g$ as, say, $v=ga$ for some $a\colon Z\to Y$. In particular,
we see that $g=guu^{-1}=vfu^{-1}=gafu^{-1}$ and hence $afu^{-1}$
is an automorphism of $Y$ as $g$ is right minimal. This means that
$af$ is also an automorphism of $Y$ and that $a$ is a retraction.
Then, by Lemma \ref{lem:add cat with split idems, X indecomp and Y neq 0, then f X to Y retraction implies f isomorphism - and dual statement},
we have that $a$ is an isomorphism because $Z$ is indecomposable
and $Y\neq0$. However, this yields that $f=a^{-1}af$ is an isomorphism,
and hence a retraction, which contradicts that $f$ is right almost
split. Hence, $v$ must be a retraction, and thus also an isomorphism
by Lemma \ref{lem:add cat with split idems, X indecomp and Y neq 0, then f X to Y retraction implies f isomorphism - and dual statement}. \phantom{spacefiller}
\end{proof}
Now we generalise \cite[Thm. 1.4]{LiuS-AR-theory-in-KS-cat}
to a not necessarily $\Hom$-finite (but still Krull-Schmidt) setting.
\begin{thm}
\label{thm:Liu thm 1.4}Let $\CA$ be a Krull-Schmidt category, and
suppose $X\overset{f}{\longrightarrow}Y\overset{g}{\longrightarrow}Z$
is an Auslander-Reiten sequence in $\CA$ with $Y\neq0$.\emph{\begin{enumerate}[(i)]
\item \emph{Up to isomorphism,} $X\overset{f}{\longrightarrow}Y\overset{g}{\longrightarrow}Z$ \emph{is the unique Auslander-Reiten sequence starting at} $X$ \emph{and the unique one ending at} $Z$\emph{.}
\item \emph{Any irreducible morphism} $f_1\colon X \to Y_1$ \emph{or} $g_1\colon Y_1 \to Z$ \emph{fits into an Auslander-Reiten sequence} $\begin{tikzcd}[column sep=1.2cm, ampersand replacement=\&] X \arrow{r}{\begin{psmallmatrix}f_1 \\ f_2\end{psmallmatrix}} \& Y_1 \oplus Y_2 \arrow{r}{\begin{psmallmatrix}g_1 & g_2\end{psmallmatrix}}\& Z.\end{tikzcd}$
\end{enumerate}}
\end{thm}

\begin{proof}
Follow the proof in \cite{LiuS-AR-theory-in-KS-cat},
replacing the use of \cite[Lem. 1.2]{LiuS-AR-theory-in-KS-cat}
with Lemma \ref{lem:Liu lem 1.2}.
\end{proof}
For a $\Hom$-finite, Krull-Schmidt category, Liu identifies a nice
class of ideals---admissible ideals. It is observed in
\cite{LiuS-AR-theory-in-KS-cat} that, for such an
ideal $\CI$ of a $\Hom$-finite, Krull-Schmidt category $\CA$, irreducible
morphisms (between indecomposables) and minimal left/right almost
split morphisms remain, respectively, so under the quotient functor
$\CA\to\CA/\CI$. We adopt the same definition but without the $\Hom$-finite
restriction.
\begin{defn}
\label{def:admissible ideal of morphisms}\cite[Def. 1.6]{LiuS-AR-theory-in-KS-cat}
Suppose $\CA$ is a Krull-Schmidt $\ring$-category. An ideal $\CI$ of
$\CA$ is called \emph{admissible} if it satisfies the following.\begin{enumerate}[(i)]
\item Whenever $X,Y\in \CA$ are indecomposable such that $1_X\notin\CI(X,X)$ and $1_Y\notin\CI(Y,Y)$, then $\CI(X,Y)\subseteq\rad^{2}_{\CA}(X,Y).$
\item If $f\colon X\to Y$ is minimal left almost split, where $1_X\notin\CI(X,X),$ and $g\in\CI(X,M),$ then we can express $g=hf$ for some $h\in\CI(Y,M).$
\item If $f\colon X\to Y$ is minimal right almost split, where $1_Y\notin\CI(Y,Y),$ and $g\in\CI(M,Y),$ then we can express $g=fh$ for some $h\in\CI(M,X).$
\end{enumerate}
\end{defn}

\begin{example}
\label{exa:morphisms factoring through subcate that is closed under summands is admissible}Suppose $\CB\subseteq\CA$ is a full subcategory closed under direct sums and direct summands.
Then the ideal $[\CB]$ of morphisms factoring through objects of
$\CB$ is admissible. See \cite[Prop. 1.9]{LiuS-AR-theory-in-KS-cat}.
\end{example}

The next result follows quickly from the definition of an admissible
ideal.
\begin{lem}
\label{lem:I admissible ideal, X,Y cap I is zero then I(X,Y) subset of rad^2(X,Y)}Suppose
$\CI$ is an admissible ideal of $\CA$. Suppose $X=\bigoplus_{i=1}^{n}X_{i}$
and $Y=\bigoplus_{j=1}^{m}Y_{j}$ are decompositions into indecomposables
in $\CA$ with $1_{X_{i}}\notin\CI(X_{i},X_{i}),1_{Y_{j}}\notin\CI(Y_{j},Y_{j})$
for all $i,j$. Then $\CI(X,Y)\subseteq\rad_{\CA}^{2}(X,Y)$.
\end{lem}

\begin{proof}
Let $f\in\CI(X,Y)$ be arbitrary and write$$f=\begin{pmatrix} f_{11} & \cdots & f_{1n} \\ \vdots & \ddots & \vdots \\ f_{m1} & \cdots & f_{mn} \end{pmatrix}$$where
$f_{ji}\colon X_{i}\to Y_{j}$. Then for each $i,j$ we have $f_{ji}=\pi_{j}f\iota_{i}\in\CI(X_{i},Y_{j})$,
where $\pi_{j}\colon Y\to Y_{j}$ is the natural projection and $\iota_{i}\colon X_{i}\to X$
is the natural inclusion. Since $\CI$ is admissible and $1_{X_{i}}\notin\CI(X_{i},X_{i}),1_{Y_{j}}\notin\CI(Y_{j},Y_{j})$,
we have that $f_{ji}\in\CI(X_{i},Y_{j})\subseteq\rad_{\CA}^{2}(X_{i},Y_{j})$
for each $i,j$. Therefore, $f$ is a sum of morphisms in $\rad_{\CA}^{2}(X,Y)$
and hence $f\in\rad_{\CA}^{2}(X,Y)$ as desired.
\end{proof}
The next lemma generalises \cite[Lem. 1.7 (1)]{LiuS-AR-theory-in-KS-cat}.
The proof in \cite{LiuS-AR-theory-in-KS-cat} makes
use of a specific characterisation of irreducible morphisms between
indecomposables (see \cite[Prop. 2.4]{Bautista-irred-morphs-and-radical-of-a-cat}),
which we cannot use since we make no indecomposability assumptions
on the domain and codomain of the morphism. See also \cite[Prop. 1.2]{AuslanderReiten-Rep-theory-of-Artin-algebras-V}.
\begin{prop}
\label{prop:I admissible, X,Y cap I is zero then f irred iff f mod I irred}Suppose
$\CI$ is an admissible ideal of $\CA$. Suppose $X=\bigoplus_{i=1}^{n}X_{i}$
and $Y=\bigoplus_{j=1}^{m}Y_{j}$ are decompositions into indecomposables
in $\CA$ with $1_{X_{i}}\notin\CI(X_{i},X_{i}),1_{Y_{j}}\notin\CI(Y_{j},Y_{j})$
for all $i,j$. Then $f\colon X\to Y$ is irreducible in $\CA$ if
and only if $\overline{f}=f+\CI(X,Y)$ is irreducible in $\CA/\CI$.
\end{prop}

\begin{proof}
$(\Rightarrow)$ Assume $f\colon X\to Y$ is an irreducible morphism
in $\CA$. By Lemma \ref{lem:X cap I is zero then f X to Y in A is sect iff f mod I is sect, and Y cap I is zero then f is ret iff f mod I ret},
$\overline{f}$ is neither a section nor a retraction in $\CA/\CI$.
Now suppose $\overline{f}=\overline{hg}$ in $\CA/\CI$ for some morphisms
$g\colon X\to Z$, $h\colon Z\to Y$ of $\CA$. Then $f-hg\in\CI(X,Y)\subseteq\rad_{\CA}^{2}(X,Y)$
by Lemma \ref{lem:I admissible ideal, X,Y cap I is zero then I(X,Y) subset of rad^2(X,Y)}.
Therefore, there is an object $W\in\CA$ and morphisms $a\in\rad_{\CA}(X,W),b\in\rad_{\CA}(W,Y)$
such that $f-hg=ba$. This yields $f=hg+ba=(\,h\;\,b\,)\begin{psmallmatrix}g\\ a\end{psmallmatrix},$
so that either $(\,h\;\,b\,)$ is a retraction or $\begin{psmallmatrix}g \\ a\end{psmallmatrix}$
is a section because $f$ is irreducible. First, assume $(\,h\;\,b\,)$
is a retraction. Then there is a morphism $\begin{psmallmatrix}s\\ t\end{psmallmatrix}\colon Y\to Z\oplus W$
such that $1_Y=(\,h\;\,b\,)\begin{psmallmatrix}s\\ t\end{psmallmatrix}=hs+bt.$
Now $b\in\rad_{\CA}(W,Y)$, so $bt\in\rad_{\CA}(Y,Y)$ as $\rad_{\CA}$
is an ideal of $\CA$. Then $hs=1_{Y}-bt$ is invertible, so $h$
is a retraction and hence $\overline{h}$ is also a retraction. In
the other case, we find that $\overline{g}$ is a section in a similar
fashion. Thus, $\overline{f}$ is an irreducible morphism.

$(\Leftarrow)$ Conversely, suppose $\overline{f}\colon X\to Y$ is
irreducible in $\CA/\CI$. By Lemma \ref{lem:X cap I is zero then f X to Y in A is sect iff f mod I is sect, and Y cap I is zero then f is ret iff f mod I ret},
$f$ cannot be a section or a retraction. Assume $f=hg$ for some
$g\colon X\to Z$, $h\colon Z\to Y$ of $\CA$. Then in $\CA/\CI$
we have $\overline{f}=\overline{hg}$ and so either $\overline{h}$
is a retraction or $\overline{g}$ is a section, since $\overline{f}$
is irreducible. Therefore, $h$ is a retraction or $g$ is a section,
respectively, by Lemma \ref{lem:X cap I is zero then f X to Y in A is sect iff f mod I is sect, and Y cap I is zero then f is ret iff f mod I ret}
again. Hence, $f$ is irreducible.
\end{proof}
In a not necessarily $\Hom$-finite, Krull-Schmidt category, the results
\cite[Lem. 1.7 (2), (3)]{LiuS-AR-theory-in-KS-cat},
\cite[Prop. 1.8]{LiuS-AR-theory-in-KS-cat} and \cite[Lem. 1.9 (2)]{LiuS-AR-theory-in-KS-cat}
all hold using the same proofs that Liu provides. This concludes our work
on generalisations of results of Liu. We now recall some last definitions
from \cite{LiuS-AR-theory-in-KS-cat} and prove some
new results.
\begin{defn}
\cite[Def. 2.2]{LiuS-AR-theory-in-KS-cat} An object
$X\in\CA$ is called \emph{pseudo-projective} (respectively, \emph{pseudo-injective})
if there exists a minimal right almost split monomorphism $W\to X$
(respectively, minimal left almost split epimorphism $X\to Y$).
\end{defn}

\begin{defn}
\label{def:(left and right) Auslander-Reiten category}\cite[Def. 2.6]{LiuS-AR-theory-in-KS-cat}
Suppose $\CA$ is a Krull-Schmidt $\ring$-category. We call $\CA$ a
\emph{left Auslander-Reiten} category if, for every indecomposable
$Z\in\CA$, either $Z$ is pseudo-projective or it is the last term
of an Auslander-Reiten sequence in $\CA$. Dually, $\CA$ is a \emph{right
Auslander-Reiten} category if, for every indecomposable $X\in\CA$,
either $X$ is pseudo-injective or it is the first term of an Auslander-Reiten
sequence. If $\CA$ is both a left and right Auslander-Reiten category,
then we simply call $\CA$ an \emph{Auslander-Reiten} category.
\end{defn}

\begin{rem}
Let $\CC$ be a triangulated category with suspension functor $\sus$. Then $\CC$ is said to \emph{have Auslander-Reiten triangles}
if for every indecomposable $Z$ there is an Auslander-Reiten triangle
ending at $Z$ (see \cite[p. 31]{Happel-triangulated-cats-in-rep-theory}). That
is, for each indecomposable $Z$ there is a triangle $X\overset{f}{\to}Y\overset{g}{\to}Z\overset{h}{\to}\sus X$
with $f$ minimal left almost split and $g$ minimal right almost
split. Therefore, a Krull-Schmidt, $\Hom$-finite, triangulated $\ring$-category
that has Auslander-Reiten triangles is immediately seen to be a left
Auslander-Reiten category in light of a result of Liu: \cite[Lem. 6.1]{LiuS-AR-theory-in-KS-cat}
shows that $X\overset{f}{\to}Y\overset{g}{\to}Z\overset{h}{\to}\sus X$
is an Auslander-Reiten triangle if and only if $X\overset{f}{\to}Y\overset{g}{\to}Z$
is an Auslander-Reiten sequence as in Definition \ref{def:AR-sequence in an additive category}.
The $\Hom$-finite assumption may be removed by noting that one can
use Theorem \ref{thm:Liu thm 1.4} in the proof of \cite[Lem. 6.1]{LiuS-AR-theory-in-KS-cat}.
\end{rem}

The following two propositions generalise \cite[Prop. 1.2]{AuslanderReiten-Rep-theory-of-Artin-algebras-V},
and the last theorem of this section is an analogue of \cite[Thm. IV.1.13]{AssemSimsonSkowronski-Vol1}
(see also \cite[Thm. 2.14]{AuslanderReiten-Rep-theory-of-Artin-algebras-IV}).
For the most part, the proofs are straightforward generalisations
of those for the abelian case, using the more general results from
this article and \cite{LiuS-AR-theory-in-KS-cat} as
appropriate. Thus, we only outline the proofs indicating the required
generalised results where it is clear what needs to be done, and provide
more details otherwise.
\begin{prop}
\label{prop:KS left AR cat, I admissible, X,Y cap I is zero, f bar X to Y is irred and right almost split in quotient iff exists X' in I such that X+X' to Y is minimal right almost split}Suppose
$\CA$ is a left Auslander-Reiten category. Let $f\colon X\to Y$
be a morphism in $\CA$ and let $\CI$ be an admissible ideal of $\CA$.
Suppose $X=\bigoplus_{i=1}^{n}X_{i}$ and $Y=\bigoplus_{j=1}^{m}Y_{j}$
are decompositions into indecomposables in $\CA$ with $1_{X_{i}}\notin\CI(X_{i},X_{i}),1_{Y_{j}}\notin\CI(Y_{j},Y_{j})$
for all $i,j$. Then $\overline{f}=f+\CI(X,Y)\colon X\to Y$ is irreducible
and right almost split in $\CA/\CI$, if and only if there exists
$g\colon X'\to Y$ in $\CA$ with $1_{X'}\in\CI(X',X')$ such that
$(\,f\,\;g\,)\colon X\oplus X'\to Y$ is minimal right almost split
in $\CA$.
\end{prop}

\begin{proof}
$(\Rightarrow)$ Use: Proposition \ref{prop:I admissible, X,Y cap I is zero then f irred iff f mod I irred}
instead of \cite[Prop. 1.2 (a)]{AuslanderReiten-Rep-theory-of-Artin-algebras-V}
to show $f$ is irreducible; \cite[Lem. 2.3]{AuslanderReiten-Rep-theory-of-Artin-algebras-IV}
and Lemma \ref{lem:Let I be an ideal of additive A, X cap I is 0 implies EndX in A local iff EndX in A mod I local}
to show $\End_{\CA}Y$ is local; and \cite[Thm. 2.4]{AuslanderReiten-Rep-theory-of-Artin-algebras-IV}
and that $\CA$ is a left Auslander-Reiten category to obtain a minimal
right almost split morphism $(\,f\,\;g\,)\colon X\oplus X'\to Y$.

By \cite[Lem. 1.7]{LiuS-AR-theory-in-KS-cat}, the
morphism $(\,\overline{f}\,\;\overline{g}\,)\colon X\oplus X'\to Y$
is minimal right almost split, and hence a non-retraction, in $\CA/\CI$.
Since $\overline{f}\colon X\to Y$ is right almost split, there exists
$(\,\overline{a}\,\;\overline{b}\,)\colon X\oplus X'\to X$ such that
$(\,\overline{fa}\,\;\overline{fb}\,)=\overline{f}\circ(\,\overline{a}\,\;\overline{b}\,)=(\,\overline{f}\,\;\overline{g}\,)$.
We now deviate from the proof given in \cite{AuslanderReiten-Rep-theory-of-Artin-algebras-V}.
This implies $$(\,\overline{f}\,\;\overline{g}\,)\begin{pmatrix}\overline{a} & \overline{b} \\ 0 & 0\end{pmatrix}=(\,\overline{fa}\,\;\overline{fb}\,)=(\,\overline{f}\,\;\overline{g}\,),$$so
$\begin{psmallmatrix}\overline{a} & \overline{b} \\ 0 & 0\end{psmallmatrix}$
is an automorphism of $X\oplus X'$ in $\CA/\CI$ as $(\,\overline{f}\,\;\overline{g}\,)$
is right minimal. Hence, there is $\begin{psmallmatrix}\overline{r} & \overline{s} \\ \overline{t} & \overline{u}\end{psmallmatrix}\in\End_{\CA/\CI}(X\oplus X')$
such that $$\begin{pmatrix}\overline{r} & \overline{s} \\ \overline{t} & \overline{u}\end{pmatrix}\begin{pmatrix}\overline{a} & \overline{b} \\ 0 & 0\end{pmatrix}=\begin{pmatrix}1_X & 0 \\ 0 & 1_{X'}\end{pmatrix}=\begin{pmatrix}\overline{a} & \overline{b} \\ 0 & 0\end{pmatrix}\begin{pmatrix}\overline{r} & \overline{s} \\ \overline{t} & \overline{u}\end{pmatrix}=\begin{pmatrix}\overline{ar}+\overline{bt} & \overline{as}+\overline{bu} \\ 0 & 0\end{pmatrix}.$$Therefore,
$\overline{1_{X'}}=0$ and hence $1_{X'}\in\CI(X',X')$.

$(\Leftarrow)$ Use \cite[Lem. 1.7]{LiuS-AR-theory-in-KS-cat}
to get that $(\,\overline{f}\,\;\overline{g}\,)$ is minimal right
almost split in $\CA/\CI$, and that $\overline{\iota_{X}}\colon X\into X\oplus X'$
is an isomorphism in the factor category as $1_{X'}\in\CI(X',X')$.
\end{proof}
Dually, the following is also true.
\begin{prop}
\label{prop:KS right AR cat, I admissible, X,Y cap I is zero, f bar X to Y is irred and left almost split in quotient iff exists Y' in I such that X to Y+Y' is minimal left almost split}Suppose
$\CA$ is a right Auslander-Reiten category. Let $f\colon X\to Y$
be a morphism in $\CA$ and let $\CI$ be an admissible ideal of $\CA$.
Suppose $X=\bigoplus_{i=1}^{n}X_{i}$ and $Y=\bigoplus_{j=1}^{m}Y_{j}$
are decompositions into indecomposables in $\CA$ with $1_{X_{i}}\notin\CI(X_{i},X_{i}),1_{Y_{j}}\notin\CI(Y_{j},Y_{j})$
for all $i,j$.  Then $f+\CI(X,Y)\colon X\to Y$ is irreducible and
left almost split in $\CA/\CI$, if and only if there exists $g\colon X\to Y'$
in $\CA$ with $1_{Y'}\in\CI(Y',Y')$ such that $\begin{psmallmatrix}f\\g\end{psmallmatrix}\colon X\to Y\oplus Y'$
is minimal left almost split in $\CA$.
\end{prop}

Our main result of this section is the following characterisation
of Auslander-Reiten sequences, which is a more general version of
\cite[Thm. IV.1.13]{AssemSimsonSkowronski-Vol1}. Furthermore, statement
(f) in \cite{AssemSimsonSkowronski-Vol1} has stronger assumptions than
the corresponding statement (vi) below: more precisely, in (vi) we
do not assume any indecomposability assumptions on the first and last
term of the short exact sequence.
\begin{thm}
\label{thm: ASS IV.1.13 for KS quasi-abelian cat - xi X to Y to Z exact seq AR iff EndX local and g right almost split iff EndZ local and f right almost split iff f min left almost split iff g min right almost split iff EndX,EndZ local and f,g irreducible}Let
$\CA$ be a skeletally small, preabelian category. Let $\xi\colon X\overset{f}{\to}Y\overset{g}{\to}Z$
be a stable exact sequence in $\CA$, i.e. $\xi\in\Ext_{\CA}^{1}(Z,X)$.
Then statements \emph{(i)--(iii)} are equivalent.\emph{\begin{enumerate}[(i)]
\item $\xi$ \emph{is an Auslander-Reiten sequence.}
\item $\End_{\CA}(X)$ \emph{is local and} $g$ \emph{is right almost split.}
\item $\End_{\CA}(Z)$ \emph{is local and} $f$ \emph{is left almost split.}
\end{enumerate}}Suppose further that $\CA$ is quasi-abelian and Krull-Schmidt. Then\emph{
(i)--(vi)} are equivalent.\emph{\begin{enumerate}[(i)]
\setcounter{enumi}{3} 
\item $f$ \emph{is minimal left almost split.}
\item $g$ \emph{is minimal right almost split.}
\item $f$ and $g$ \emph{are irreducible.}
\end{enumerate}}
\end{thm}

\begin{proof}
From Definition \ref{def:AR-sequence in an additive category} and
\cite[Lem. 2.3]{AuslanderReiten-Rep-theory-of-Artin-algebras-IV},
(ii) and (iii) follow from (i). To show (ii) $\Rightarrow$ (iii)
and (iii) $\Rightarrow$ (i), use Proposition \ref{prop:in preabelian cat endo u v w of non-split stable short exact seq X to Y to Z with EndX (EndZ) local and w (u) iso implies u (w) iso and v iso too}
instead of \cite[Lem. IV.1.12]{AssemSimsonSkowronski-Vol1}. And (iii)
$\Rightarrow$ (ii) is dual to (ii) $\Rightarrow$ (iii), so this
establishes the equivalence of (i)--(iii).

Now suppose further that $\CA$ is quasi-abelian and Krull-Schmidt.
Statements (iv) and (v) follow from (i) by definition, and (iv) $\Rightarrow$
(iii) is dual to (v) $\Rightarrow$ (ii).

First, we claim that if $g$ is right almost split then $Y$ is non-zero.
Indeed, if $Y=0$ then $1_{Z}\circ g=g=0$, which implies $1_{Z}=0$
as $g=\cok f$ is an epimorphism (since $\xi$ is short exact). However,
if $g$ is right almost split,  then $\End_{\CA}Z$ is local by \cite[Lem. 2.3]{AuslanderReiten-Rep-theory-of-Artin-algebras-IV}
and hence $1_{Z}$ cannot be the zero morphism.

(v) $\Rightarrow$ (ii). Since $g$ is right almost split, we may
use our claim above to conclude that $g$ is irreducible by Proposition
\ref{prop:left (right) minimal almost split with nonzero target (domain) is irreducible}
(ii). Then $X$ is indecomposable by Proposition \ref{prop:A semi-abelian xi X f Y g Z short exact - f irred implies xi min and Z indecomp - and dually for g},
which is equivalent to $\End_{\CA}X$ being local as $\CA$ is Krull-Schmidt.

For (i) implies (vi), use Proposition \ref{prop:left (right) minimal almost split with nonzero target (domain) is irreducible}
(noting again that $Y$ is non-zero if $g$ is right almost split).

(vi) $\Rightarrow$ (ii). Suppose that $f,g$ are irreducible. First
we show that $g$ is right almost split. Note that $g$ is not a retraction
as it is irreducible by assumption. Thus, let $h\colon M\to Z$ be
a non-retraction. Since $\CA$ is Krull-Schmidt we may write $M=\bigoplus_{i=1}^{n}M_{i}$,
for some indecomposable objects $M_{i}$, and $h=(\,h_{1}\,\;\cdots\,\;h_{n}\,)$
where $h_{i}\colon M_{i}\to Z$. Since $h$ is not a retraction, it
follows that no $h_{i}$ may be a retraction either. Fix $i\in\{1,\ldots,n\}$.
As $f$ is irreducible, the criterion from Proposition \ref{prop:ASS Lem IV.1.7(a) or AR IV Prop 2.7 (a) for quasi-abelian category - criterion for irreducible if f is a kernel}
tells us that either there exists $v_{i,1}\colon M_{i}\to Y$ such
that $gv_{i,1}=h_{i}$ or there exists $v_{i,2}\colon Y\to M_{i}$
such that $g=h_{i}v_{i,2}$. Suppose we are in the latter case and
that $g=h_{i}v_{i,2}$ for some $v_{i,2}\colon Y\to M_{i}$. Then,
as $g$ is irreducible and $h_{i}$ is not a retraction, we have that
$v_{i,2}$ is section. But $M_{i}$ is indecomposable and $Y\neq0$,
so $v_{i,2}$ is in fact an isomorphism by Lemma \ref{lem:add cat with split idems, X indecomp and Y neq 0, then f X to Y retraction implies f isomorphism - and dual statement}.
In this case, we then get $h_i=g\circ\tensor[]{v}{_{i,2}^{-1}}$.
Therefore, for all $1\leq i\leq n$ we have that $h_{i}=g\circ w_{i}$
for some $w_{i}\colon M_{i}\to Y$. Hence, $h=(\,h_{1}\,\;\cdots\,\;h_{n}\,)=g\circ(\,w_{1}\,\;\cdots\,\;w_{n}\,)$
and $g$ is seen to be right almost split. Dually, we have that $f$
is left almost split and hence $\End_{\CA}X$ is local by \cite[Lem. 2.3]{AuslanderReiten-Rep-theory-of-Artin-algebras-IV}.

This shows (i)--(vi) are equivalent and finishes the proof.
\end{proof}

We conclude this section with an example of a $\Hom$-infinite, Krull-Schmidt category. The author is grateful to P.-G. Plamondon for communicating the following example and answering several questions.

\begin{example}\label{exa:Markov-quiver}
Let $k$ be a field. Consider the quiver with potential $(Q,W)$ where $Q$ is the quiver 
%$$\begin{tikzcd}[column sep=1cm, row sep=1cm] &1\arrow{dr}[swap]{a'}\arrow[shift left=0.5em]{dr}{a}&\\ 3\arrow[shift left=0.5em]{ur}{c}\arrow{ur}[swap]{c'} &&2 \arrow[shift left=0.25em]{ll}{b}\arrow[shift right=0.25em]{ll}[swap]{b'}
%\end{tikzcd}$$ 
$$\begin{tikzpicture}%[baseline=-0.4ex]
\node (1) at (2,2) {\Large 1};
\node (2) at (4,0) {\Large 2};
\node (3) at (0,0) {\Large 3};
\node (2w) at (3.6,0) {};
\node (3e) at (0.4,0) {};

\node (2nw) at (4,0.2) {};
\node (3ne) at (0,0.2) {};

\node (1se) at (2.2,1.9) {};
\node (1sw) at (1.8,1.9) {};

\draw[->, transform canvas={xshift=0.5em, yshift=0.3em}, font=\tiny]
(1se) edge node[auto,fill=white, anchor=center, pos=0.5] {$a$} (2nw);
\draw[->, transform canvas={xshift=-0.35em, yshift=-0.3em}, font=\tiny]
(1se) edge node[auto,fill=white, anchor=center, pos=0.5] {$a'$} (2nw);

\draw[->, transform canvas={yshift=0.3em}, font=\tiny]
(2w) edge node[auto,fill=white, anchor=center, pos=0.5] {$b'$} (3e);
\draw[->, transform canvas={yshift=-0.6em}, font=\tiny]
(2w) edge node[auto,fill=white, anchor=center, pos=0.5] {$b$} (3e);

\draw[->, transform canvas={xshift=-0.3em, yshift=0.4em}, font=\tiny]
(3ne) edge node[auto,fill=white, anchor=center, pos=0.5] {$c$} (1sw); 
\draw[->, transform canvas={xshift=0.5em, yshift=-0.3em}, font=\tiny]
(3ne) edge node[auto,fill=white, anchor=center, pos=0.5] {$c'$} (1sw); 

\end{tikzpicture}$$
and $W=cba+c'b'a'$ is the potential. Following \cite[\S 2.6]{KellerYang-derived-equivalences-from-mutations-of-quivers-with-potential}, we recall the construction of the \emph{complete Ginzburg dg algebra} $G\deff\wh{\Gamma}(Q,W)$ associated to $Q$. From $Q$, consider the quiver $\wt{Q}$: 
$$\begin{tikzpicture}%[baseline=-0.4ex]
\node (1) at (2.5,2.5) {\Large 1};
\node (2) at (5,0) {\Large 2};
\node (3) at (0,0) {\Large 3};
\node (2w) at (4.7,0) {};
\node (2sw) at (4.95,-0.2) {};
\node (2s) at (5.1,-0.2) {};
\node (2nw) at (5,0.2) {};
\node (2e) at (5,0) {};
\node (2n) at (5.0,0.2) {};
\node (3e) at (0.3,0) {};
\node (3se) at (0.0,-0.2) {};
\node (3s) at (0.0,-0.2) {};
\node (1se) at (2.8,2.2) {};
\node (1e) at (2.5,2.2) {};
\node (3n) at (0,0.3) {};
\node (1sw) at (2.3,2.4) {};

\draw[->, transform canvas={xshift=0.5em, yshift=0.3em}, font=\tiny]
(1se) edge node[auto,fill=white, anchor=center, pos=0.5] {$a$} (2nw);
\draw[->, transform canvas={xshift=-0.4em, yshift=-0.3em}, font=\tiny]
(1se) edge node[auto,fill=white, anchor=center, pos=0.5] {$a'$} (2nw);

\draw[->, transform canvas={yshift=0.3em}, font=\tiny]
(2w) edge node[auto,fill=white, anchor=center, pos=0.5] {$b'$} (3e);
\draw[->, transform canvas={yshift=-0.6em}, font=\tiny]
(2w) edge node[auto,fill=white, anchor=center, pos=0.5] {$b$} (3e);

\draw[->, transform canvas={xshift=-0.15em, yshift=0.15em}, font=\tiny]
(3ne) edge node[auto,fill=white, anchor=center, pos=0.5] {$c$} (1sw); 
\draw[->, transform canvas={xshift=0.6em, yshift=-0.5em}, font=\tiny]
(3ne) edge node[auto,fill=white, anchor=center, pos=0.5] {$c'$} (1sw); 
    
\path[->, transform canvas={xshift=0.5em, yshift=0.5em}, font=\tiny, bend right=90, looseness=1.2]
(2e) edge node[auto,fill=white, anchor=center, pos=0.5] {$a^*$} (1e);
\path[->, transform canvas={xshift=0.5em, yshift=0.3em}, font=\tiny,bend right=50]
(2n) edge node[auto,fill=white, anchor=center, pos=0.5] {$a'^*$} (1e);

\draw[->, transform canvas={xshift=0em,yshift=-0.0em}, font=\tiny, bend right=30]
(3se) edge node[auto,fill=white, anchor=center, pos=0.5] {$b'^*$} (2sw);
\draw[->, transform canvas={xshift=0em,yshift=-0.0em}, font=\tiny, bend right=50]
(3s) edge node[auto,fill=white, anchor=center, pos=0.5] {$b^*$} (2s);

\draw[->, transform canvas={xshift=-0.0em, yshift=0.0em}, font=\tiny, bend right=90,looseness=1.2]
(1sw) edge node[auto,fill=white, anchor=center, pos=0.5] {$c^*$} (3n); 
\draw[->, transform canvas={xshift=-0em, yshift=-0.0em}, font=\tiny, bend right=50]
(1sw) edge node[auto,fill=white, anchor=center, pos=0.5] {$c'^*$} (3n); 

\path[->, font=\tiny,]
(1) edge [auto, in=60, out=120, looseness=8] node[above] {$t_1$} (1);

\path[->, font=\tiny,]
(2) edge [auto, in=300, out=0, looseness=8] node[above, xshift=0.6em, yshift=-0.3em] {$t_2$} (2);

\path[->, font=\tiny,]
(3) edge [auto, in=180, out=240, looseness=8] node[above, xshift=-0.5em, yshift=-0.3em] {$t_3$} (3);
\end{tikzpicture}$$ 
The quiver $\wt{Q}$ is given the following grading: arrows $x, x'$ have degree $0$ and arrows $x^*, x'^*$ have degree $-1$ for $x\in\{a,b,c\}$, and the loop $t_i$ has degree $-2$ for $1\leq i\leq 3$. Then $G$ has underlying graded algebra given by the completion of the graded path algebra $k\wt{Q}$ with respect to the ideal generated by the arrows of $\wt{Q}$ in the category of graded $k$-vector spaces. Furthermore, $G$ is a dg algebra, equipped with a differential of degree $+1$. 
%$d$, satisfying: $dx=0, dx=x'$ for $x\in\{a,b,c\}$; 
%$$d(x^*)=\partial_{x}(W)\deff \sum_{W=yxz}zy,$$ 
%where the sum is over all decompositions of $W$ with $y,z$ (possibly trivial) paths, and similarly $d(x'^*)=\partial_{x'}(W)$ for $x\in\{a,b,c\}$; and 
%$$d(t_i) = e_i\left(\sum_{x\in\{a,b,c\}}([x,x^*]+[x',x'^*])\right)e_i$$ for $1\leq i\leq 3$. 

Let $\rmod{G}$, $\kom(G)$ and $\der(G)$ denote the category of right dg $G$-modules, the homotopy category of right dg $G$-modules and the corresponding derived category, respectively. The \emph{perfect derived category} $\per G$ is the smallest, full, subcategory of $\der(G)$ that contains $G$, and is closed under shifts, extensions and direct summands. Let $J(Q,W)$ denote the \emph{Jacobian algebra} associated to $(Q,W)$. Then $J(Q,W)$ is the complete path algebra $\wh{kQ}$ modulo the closure of the ideal generated by $\partial_x(W)$ and $\partial_{x'}(W)$ for $x\in\{a,b,c\}$, where 
$$\partial_{x}(W)\deff \sum_{W=yxz}zy,$$ 
where the sum is over all decompositions of $W$ with $y,z$ (possibly trivial) paths. The sum $\partial_{x'}(W)$ is defined similarly. It is easy to check that $J(Q,W)$ is infinite-dimensional over $k$.

The category $\per G$ is Krull-Schmidt by \cite[Lem. 2.17]{KellerYang-derived-equivalences-from-mutations-of-quivers-with-potential}. Furthermore, we have  
$$\def\arraystretch{1.5}\begin{array}{rcll}
\End_{\per G}G &=&\End_{\der(G)}G& \text{since }\per G\text{ is a full subcategory}\\
&\iso&\Hom_{\kom(G)}(G,G)&\text{since }G\text{ is cofibrant (see \cite[pp. 2126--2127]{KellerYang-derived-equivalences-from-mutations-of-quivers-with-potential})}\\
&=&H^{0}(\CH om_{\rmod{G}}(G,G))&\\
&\iso&H^{0}(G)&\\
&=&J(Q,W)&\text{by \cite[Lem. 2.8]{KellerYang-derived-equivalences-from-mutations-of-quivers-with-potential}}.
\end{array}$$ 
It follows that $\per G$ is a $\Hom$-infinite $k$-category and is also Krull-Schmidt. We remark that, by \cite[Lem. 2.9]{Plamondon-cluster-characters-for-cluster-categories-with-infinite-dimensional-morphism-spaces}, the corresponding cluster category is also $\Hom$-infinite $k$-category in this case.
\end{example}

%%%%%%%%%%%%%%%%%%%%%%%%%%%%%%%%%%%%%%%%%%%%
\section{\label{sec:example}An example from cluster theory}

We now present an example coming from cluster theory that encapsulates
some of the theory we have explored.
\begin{example}
\label{exa:Q is A_3 R is P_1 oplus P_2}Let $k$ be a field. Consider the cluster category
$\CC\deff\CC_{kQ}$ (as defined in \cite{BMRRT-cluster-combinatorics})
associated to the linearly oriented Dynkin quiver 
$$Q:\quad1\to2\to3.$$ 
It is shown in \cite{BMRRT-cluster-combinatorics}
that $\CC$ is Krull-Schmidt and it is triangulated by
a result of Keller \cite{Keller-triangulated-orbit-categories}. Let $\sus$ denote the suspension functor of $\CC$. Its Auslander-Reiten
quiver, with the meshes omitted, is $$\makebox[\textwidth][c]{\includegraphics[width=0.7\textwidth]{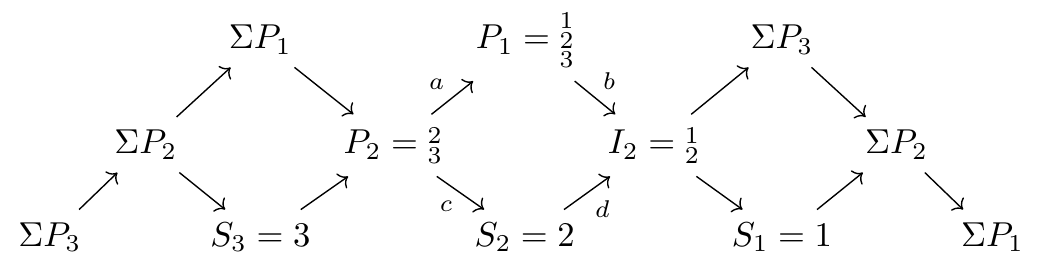}}$$where
the lefthand copy of $\sus P_{i}$ is identified with the righthand
copy (for $i=1,2,3$). We set $R\deff P_{1}\oplus P_{2}$, which is
a basic, rigid object of $\CC$. By $\add\sus R$ we denote the full subcategory of $\CC$ consisting of objects that are isomorphic
to direct summands of finite direct sums of copies of $\sus R$. The
full subcategory $\CX_{R}$ consists of objects $X$ for which $\Hom_{\CC}(R,X)=0$.
Then the pair $((\CS,\CT),(\CU,\CV))=((\add\sus R,\CX_{R}),(\XR,\add\sus R))$
is a twin cotorsion pair on $\CC$ with heart $\overline{\CH}=\CC/[\CX_{R}]$
(see \cite[Lem. 5.4]{Shah-quasi-abelian-hearts-of-twin-cotorsion-pairs-on-triangulated-cats}
and \cite[Cor. 5.9]{Shah-quasi-abelian-hearts-of-twin-cotorsion-pairs-on-triangulated-cats},
or \cite[Exa. 2.10]{Nakaoka-twin-cotorsion-pairs},
for more details), where $[\CX_{R}]$ is the ideal of morphisms factoring
through objects of $\CX_{R}$. Note that $[\CX_{R}]$ is an admissible
ideal by Example \ref{exa:morphisms factoring through subcate that is closed under summands is admissible}
as $\CX_{R}$ is closed under direct summands.

The subcategory $\CX_{R}$ is described pictorially below, where ``$\circ$''
denotes that the corresponding object does not belong to the subcategory.
\begin{gather*}
\CT = \CX_R = \CU \\ \vspace{0.6cm}
\includegraphics[width=0.7\textwidth]{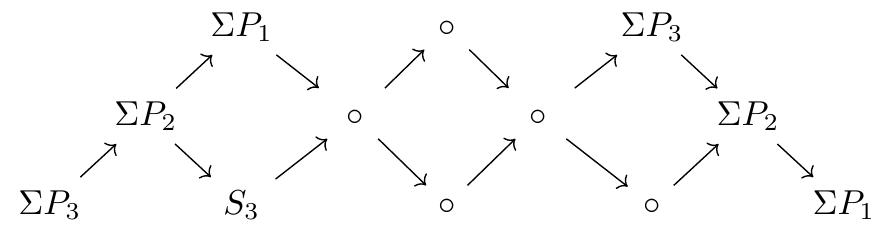}
\end{gather*}The heart $\overline{\CH}=\CC/[\CX_{R}]$ for this twin cotorsion
pair is quasi-abelian by \cite[Thm. 5.5]{Shah-quasi-abelian-hearts-of-twin-cotorsion-pairs-on-triangulated-cats},
and, by \cite[Prop. 2.9]{LiuS-AR-theory-in-KS-cat}, has the following Auslander-Reiten quiver (ignoring the objects
denoted by a ``$\circ$'' that lie in $\CX_{R}$) \begin{gather*}
 \overline{\CH}=\CC/[\CX_R]\\ \vspace{0.6cm}
\includegraphics[width=0.7\textwidth]{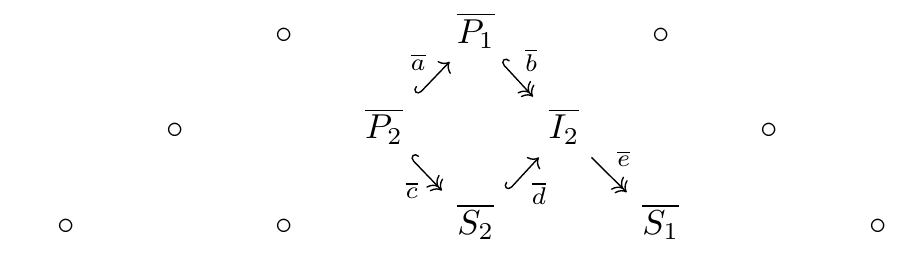}
\end{gather*}where one may define the Auslander-Reiten quiver for a Krull-Schmidt
category as in \cite{LiuS-AR-theory-in-KS-cat}. Again
we have omitted the meshes. Furthermore, we have denoted by $\overline{X}$
the image in $\CC/[\CX_{R}]$ of the object $X$ of $\CC$, monomorphisms
by ``$\hookrightarrow$'' and epimorphisms by ``$\twoheadrightarrow$''.
In this example, we notice that there are precisely two irreducible
morphisms (up to a scalar) that are regular (monic and epic simultaneously)---namely, $\overline{b}$ and $\overline{c}$.

Consider the Auslander-Reiten triangle $\begin{tikzcd}P_2 \arrow{r}{\begin{psmallmatrix}a \\ c\end{psmallmatrix}} &P_1 \oplus S_2 \arrow{r}{(\,b\,\;d\,)}& I_2 \arrow{r}{}& \sus P_2\end{tikzcd}$
in $\CC$, and note that the minimal left almost split morphism $\begin{psmallmatrix}a \\ c\end{psmallmatrix}$
is irreducible by Proposition \ref{prop:left (right) minimal almost split with nonzero target (domain) is irreducible}.
Therefore, by Proposition \ref{prop:I admissible, X,Y cap I is zero then f irred iff f mod I irred},
$\begin{psmallmatrix}\overline{a} \\ \overline{c}\end{psmallmatrix}\colon \overline{P_2}\to \overline{P_1}\oplus \overline{S_2}$
is also irreducible. Similarly, $(\,\overline{b}\,\;\overline{d}\,)\colon \overline{P_1}\oplus \overline{S_2} \to \overline{I_2}$
is irreducible in $\CC/[\CX_{R}]$. We remark that one cannot use
\cite[Lem. 1.7 (1)]{LiuS-AR-theory-in-KS-cat} since
the morphisms are not between indecomposable objects.

One can check that $\begin{psmallmatrix}\overline{a} \\ \overline{c}\end{psmallmatrix}=\ker (\,\overline{b}\,\;\overline{d}\,)$
and $(\,\overline{b}\,\;\overline{d}\,)=\cok \begin{psmallmatrix}\overline{a} \\ \overline{c}\end{psmallmatrix}$
by, for example, using the construction of (co)kernels as in \cite[Lem. 3.4]{BuanMarsh-BM2}.
So, we have that $$\begin{tikzcd}\overline{P_2} \arrow{r}{\begin{psmallmatrix}\overline{a} \\ \overline{c}\end{psmallmatrix}} &\overline{P_1} \oplus \overline{S_2} \arrow{r}{(\,\overline{b}\,\;\overline{d}\,)}& \overline{I_2}\end{tikzcd}$$is
a short exact sequence in the quasi-abelian, Krull-Schmidt category
$\CC/[\CX_{R}]$. Hence, by Theorem \ref{thm: ASS IV.1.13 for KS quasi-abelian cat - xi X to Y to Z exact seq AR iff EndX local and g right almost split iff EndZ local and f right almost split iff f min left almost split iff g min right almost split iff EndX,EndZ local and f,g irreducible},
the sequence is an Auslander-Reiten sequence because it satisfies
statement (vi) in the Theorem. Note that we could also have established
this fact using \cite[Prop. 1.8]{LiuS-AR-theory-in-KS-cat}
and Proposition \ref{prop:A preabelian, AR sequence (weak exact) is in fact exact}.

Furthermore, this example also shows that the indecomposability conditions
in Proposition \ref{prop:A quasi-abelian f irred mono indecomp target then any irred map to coker of f is epic - and dual}
cannot be removed. The morphism $\begin{psmallmatrix}\overline{a} \\ \overline{c}\end{psmallmatrix}$
is an irreducible monomorphism, but has decomposable target, and the
morphism $\overline{d}$ is an irreducible morphism with codomain
the cokernel of $\begin{psmallmatrix}\overline{a} \\ \overline{c}\end{psmallmatrix}$
that is not epic. Indeed, $\overline{e}\overline{d}=0$ but $\overline{e}\neq0$
so $\overline{d}$ cannot be an epimorphism.
\end{example}

%%%%%%%%%%%%%%%%%%%%%%%%%%%%%%%%%%%%%%%%%%%%

\begin{acknowledgements}
The author would like to thank Robert J. Marsh for his helpful guidance and support during the preparation of this article, and also the University of Leeds for financial support through a University of Leeds 110 Anniversary Research Scholarship. The author is grateful to the referee for comments on an earlier version of the paper.
\end{acknowledgements}

\bibliography{mybib}
\bibliographystyle{mybst}

\end{document}